\documentclass[a4paper,10pt,twoside]{article}
\pdfoutput=1
\title{A monodromy criterion for existence of N\'eron models of abelian schemes in characteristic zero}

\usepackage{amsmath, amssymb, mathrsfs, amsthm, enumitem}
\usepackage[numbers]{natbib}
\usepackage[all]{xy}
\usepackage{graphicx}
\usepackage{wrapfig}
\usepackage{tikz}
\usetikzlibrary{arrows}
\usepackage{caption}
\usepackage{subcaption}
\usepackage{comment}
\usetikzlibrary{positioning}
\usepackage{hyperref}
\usepackage{bookmark}
\usepackage{tikz-cd}
\usepackage{float}
\usepackage[paperwidth=170mm, paperheight=240mm, hmarginratio=5:4, top=25mm]{geometry}
\usepackage{titletoc}
\usepackage{tocloft}

\usepackage{cleveref}
\usepackage{tabularx}
\usepackage{epigraph}
\usepackage{mathtools}
\usepackage{titlesec}

\titleformat{\subsubsection}[runin]{\bfseries}{\thesubsubsection}{0.5em}{}[.]

\titleformat{\section}
  {\normalfont\scshape\centering\Large}{\thesection}{1em}{}

\titleformat{\subsection}[runin]{\bfseries}{\thesubsection}{0.5em}{}[.]

\newcommand{\note}[1]{\normalsize{{\color{red}\footnote{{\color{teal}(Giulio) #1}}}{\marginpar[{\color{teal}\hfill\tiny\thefootnote$\rightarrow$}]{{\color{teal}$\leftarrow$\tiny\thefootnote}}}}}

\newcommand{\Z}{\mathbb{Z}}

\newcommand{\R}{\mathbb{R}}
\newcommand{\Q}{\mathbb{Q}}
\newcommand{\C}{\mathbb{C}}

\newcommand{\wt}{\widetilde}

\renewcommand{\O}{\mathcal{O}}

\newcommand{\m}{\mathfrak{m}}

\newcommand{\ord}{\mathrm{ord}}

\newcommand{\bmt}{\begin{pmatrix}}
\newcommand{\emt}{\end{pmatrix}}
\newcommand{\bsm}{\left(\begin{smallmatrix}}
\newcommand{\esm}{\end{smallmatrix}\right)}

\newcommand{\til}{\widetilde}

\DeclareMathOperator{\Aut}{Aut}

\DeclareMathOperator{\End}{End}
\DeclareMathOperator{\Gal}{Gal}
\DeclareMathOperator{\Hom}{Hom}

\DeclareMathOperator{\rk}{rk}
\DeclareMathOperator{\id}{id}

\DeclareMathOperator{\im}{im}

\DeclareMathOperator{\Spec}{Spec}
\DeclareMathOperator{\colim}{colim}

\DeclareMathOperator{\coker}{coker}

\DeclareMathOperator{\Sets}{\mathbf{Sets}}
\DeclareMathOperator{\Sch}{\mathbf{Sch}}

\DeclareMathOperator{\Frac}{Frac}

\DeclareMathOperator{\Pic}{Pic}

\DeclareMathOperator{\divv}{div}

\DeclareMathOperator{\codim}{codim}

\DeclareMathOperator{\into}{\hookrightarrow}

\DeclareMathOperator{\htt}{ht}

\newcommand{\ra}{\rightarrow}

\theoremstyle{definition}
\newtheorem{definition}{Definition}[section]

\newtheorem{remark}[definition]{Remark}
\newtheorem{example}[definition]{Example}

\newtheorem{situation}[definition]{Situation}

\theoremstyle{plain}
\newtheorem{proposition}[definition]{Proposition}
\newtheorem{lemma}[definition]{Lemma}
\newtheorem{theorem}[definition]{Theorem}
\newtheorem{corollary}[definition]{Corollary}
\newtheorem{claim}[definition]{Claim}

\theoremstyle{remark}

\renewcommand{\phi}{\varphi}

\author{Giulio Orecchia \footnote{Supported by the Centre Henri Lebesgue, programme ANR-11-LABX-0020-01} \\ email: \href{mailto:giulio.orecchia@univ-rennes1.fr}{giulio.orecchia@univ-rennes1.fr} \\
Universit\'e Rennes 1}

\newcounter{nootje}
\newcommand{\beq}{\begin{equation}}
\newcommand{\eeq}{\end{equation}}
\newcommand{\beqs}{\begin{equation*}}
\newcommand{\eeqs}{\end{equation*}}

\begin{document}
\maketitle

\begin{abstract}
We consider the problem of existence of N\'eron models for a family of abelian varieties over a base of dimension greater than 1. We show that when $S$ is of equicharacteristic zero, the condition of toric additivity introduced in \citep{Orecchia} is sufficient for the existence of a N\'eron model, and also necessary under some extra assumptions. Furthermore, we give an equivalent formulation of toric additivity in terms of monodromy action on the $l$-adic Tate module.
\end{abstract}
\tableofcontents

\cleardoublepage
\pagenumbering{arabic}

\section*{Introduction}\label{section_intro1} 
\subsection*{The problem of existence of N\'eron models}
A well known result in arithmetic geometry states that every abelian variety $A$ defined over the fraction field $K$ of the spectrum $S=\Spec R$ of a Dedekind domain, admits a canonical, smooth, separated model $\mathcal N/S$, with $\mathcal N\times_SK=A$, such that every $K$-point of $A$ extends uniquely to an $S$-valued point of $\mathcal N$. Such a model $\mathcal N/S$ is called a \textit{N\'eron model} for $A$ over $S$. What makes N\'eron models particularly significant is that they inherit a unique structure of $S$-group scheme from $A$; and the fact that they give a meaning to the notion of reduction of a $K$-point of $A$ modulo a maximal ideal of $R$.

One natural question that arises is whether the theory of N\'eron models carries over when one replaces the Dedekind domain $R$ by a higher dimensional base. One first difficulty that appears, is that it is not at all clear whether N\'eron models exist in this more general setting.

The question of existence of N\'eron models has first been raised for jacobians of nodal curves in \citep{holmes}, where the author gives it a negative answer, by providing a necessary and sufficient criterion for the jacobian of a nodal family $\mathcal C/S$ of curves to admit a N\'eron model. The condition of alignment involves a rather restrictive combinatorial condition on the dual graphs of the geometric fibres of $\mathcal C/S$, endowed with a certain labelling of the edges. However, the criterion applies only when the total space $\mathcal C$ is regular.

In \citep{Orecchia}, a new criterion for existence of Neron models of jacobians was introduced, called \textit{toric additivity}. The criterion involves no condition of regularity of the total space; moreover, toric additivity is a property of the zero-component of the Picard scheme, $\Pic^0_{\mathcal C/S}$, therefore it has good base change properties, see \citep[2.3]{Orecchia}.

In fact, the zero-component $\Pic^0_{\mathcal C/S}$ is a semiabelian scheme; it seems natural to ask whether the results of \citep{Orecchia} apply more in general to abelian schemes $A/U$ admitting a semiabelian model $\mathcal A/S$; in this article we give a positive answer to the question under some extra assumptions.

\subsection*{Results}
We work over a regular base $S$, and we assume we are given an abelian scheme $A$ over an open $U\subset S$, such that $S\setminus U$ is a normal crossing divisor; we also assume that $A$ has semiabelian reduction over $S$. We call $\mathcal A/S$ the unique semiabelian scheme extending $A$. 

Over the strict henselization $S^{sh}$ at a geometric point $s$ of $S$, we define a certain \textit{purity map} (\ref{def_purity_map}) between character groups of the maximal tori contained in the fibres of $\mathcal A_{S^{sh}}$. Inspired by \citep[2.6]{Orecchia}, we say that $\mathcal A/S$ is \textit{toric additive} at $s$ if the purity map is an isomorphism.

Here is the main result of this article:
\begin{theorem}[\ref{TA->Neron}]\label{intro:mainthm}
Assume $S$ is a regular, locally noetherian $\mathbb Q$-scheme. Let $\mathcal A/S$ be a semiabelian scheme, restricting to an abelian scheme over the complement $U=S\setminus D$ of a normal crossing divisor. If $\mathcal A/S$ is toric additive at all geometric points of $S$, there exists a N\'eron model $\mathcal N/S$ for $A$ over $S$.
\end{theorem}

We introduce the notion of \textit{test-N\'eron model}; it is a smooth model $\mathcal M$ of $A$, with a structure of group-space, and with the property:  for every morphism $Z\to S$, with $Z$ the spectrum of a DVR, hitting transversally the boundary divisor $D$, the pullback $\mathcal M_Z/Z$ is a N\'eron model of its generic fibre. This notion allows us to give a partial converse to \cref{intro:mainthm}:
\begin{theorem}[\ref{thm:converse}]\label{intro:converse}
Assume $A/U$ admits a N\'eron model $\mathcal N/S$, such that $\mathcal N/S$ is also a test-N\'eron model for $A$ over $S$. Then $A/U$ is toric additive.
\end{theorem}

It is unknown to the author whether every N\'eron model is automatically a test-N\'eron model. In an forthcoming article, we show that this is true for N\'eron models of jacobians.

\Cref{intro:mainthm,intro:converse} partially generalize the previously known result from \cite{Orecchia}:
\begin{theorem}[\citep{Orecchia}, 4.13]
Let $S$ be a regular, locally noetherian, excellent scheme. Let $\mathcal C/S$ be a nodal curve, smooth over the complement $U=S\setminus D$ of a normal crossing divisor. The jacobian $\Pic^0_{\mathcal C_U/U}$ is toric additive if and only if it admits a N\'eron model over $S$.
\end{theorem}

\subsection*{Toric additivity in terms of monodromy action}
Another result of this article is that when working on a strictly local base $S$, toric additivity can be expressed in terms of the monodromy action on the Tate module of the generic fibre $A_K/K$. For a prime $l$ different from the residue characteristic $p\geq 0$, the Galois action on $T_lA(K^{sep})$ factors via the \'etale fundamental group $\pi_1(U)$, and in fact via its biggest pro-$l$ quotient, which is isomorphic to a product $G=\prod_{i=1}^n I_i$, with each $I_i=\mathbb Z_l(1)$; the indexes $i=1,\ldots,n$ correspond to the components $D_1,\ldots,D_n$ of the divisor $D=S\setminus U$.

We say that $\mathcal A/S$ is $l$-toric additive if the $G$-module $T_lA(K^{sep})/T_lA(K^{sep})^G$ decomposes into a direct sum $\bigoplus_{i=1}^n V_i$ of $G$-invariant subgroups, such that for every $i\neq j$, the subgroup $I_i\subset G$ acts as the identity on $V_j$. Roughly, one could interpret $l$-toric additivity as the fact that the groups $I_i$ act ``independently'' on $T_lA(K^{sep})$.

By itself, $l$-toric additivity is a weaker condition than toric additivity. In characteristic zero, toric additivity is equivalent to $l$-toric additivity being satisfied for all primes $l$. However, we show that in the particular case of jacobians of nodal curves, something stronger holds:
\begin{theorem}[\ref{coro:curves_l}]
Let $\mathcal A=\mathcal J_{\mathcal C/S}$ be the jacobian of a nodal curve $\mathcal C/S$, smooth over $U=S\setminus D$. The following are equivalent:
\begin{itemize}
\item[i)] $\mathcal J_{\mathcal C/S}$ is toric additive;
\item[ii)] for some prime $l\neq p$, $\mathcal J_{\mathcal C_S/S}$ is $l$-toric additive;
\item[iii)] for every prime $l\neq p$, $\mathcal J_{\mathcal C_S/S}$ is $l$-toric additive.
\end{itemize}
\end{theorem}

Somewhat surprisingly, the same result does not hold for abelian schemes. We construct an explicit example (\ref{example:uniformization}) of a semiabelian family $\mathcal A/S$, on a base of characteristic zero, that is $l$-toric additive for all primes except one, and therefore fails to be toric additive.
\subsection*{On the proof of \cref{intro:mainthm} and \cref{intro:converse}}
The proof of \cref{intro:mainthm} is constructive and substantially divided into two parts: first, starting from a toric additive $A/U$, we construct a smooth $S$-group scheme $\mathcal N/S$ extending $A$, having the following property: for all test curves $Z\to S$ hitting transversally the normal crossing divisor $D$, the restriction $\mathcal N_{|Z}$ is a N\'eron model of its generic fibre. We call such a model $\mathcal N$ a \textit{test-N\'eron model} for $A$ over $S$. In the second part of the proof, we proceed to check that a test-N\'eron model of a toric additive abelian scheme is in fact a N\'eron model. We remark that for this last fact, it is crucial that test-N\'eron are defined to be group objects; there are examples of objects that are similar to test-N\'eron models, in that they satisfy a similar property with respect to transversal traits, but fail to be a N\'eron model because they do not admit a group structure: an example is the \textit{balanced Picard stack} $\mathcal P_{d,g}\ra \overline{\mathcal M}_g$ constructed by Caporaso in \citep{caporaso}, or those constructed by Andreatta in \citep{Andreatta2001}

We mention that the characteristic zero assumption of \cref{intro:mainthm} is used only in the second part of the proof, (specificially in \cref{dim2}), and not in the construction of the test-N\'eron model. The proof uses a generalization of the result appearing in \citep{edix} on descent of N\'eron models along tamely ramified covers of discrete valuation rings. The only obstruction to removing the characteristic assumption in \cref{intro:mainthm}, appears to be the lack of results descent of N\'eron models along finite flat wildly ramified covers (at least in the case of abelian varieties with semiabelian reduction).

For what concerns \cref{intro:converse}, the assumptions allow to reduce it to a mere abstract algebra statement. One may hope that in fact every N\'eron model were automatically a test-N\'eron model, and that the extra hypothesis in the statement may be dropped. In a forthcoming article, we show that this is indeed the case for N\'eron models of jacobians of nodal curves. If a similar result were established for abelian schemes with semiabelian reduction, then toric additivity would indeed provide a necessary and sufficient criterion for existence of N\'eron models.

\subsection*{Outline}
In \cref{section1}, we follow very closely Expos\'e IX of \citep{SGA7}, \textit{Mod\`eles de N\'eron et monodromie}, where $1$-dimensional generations of abelian varieties are studied in terms of monodromy action on the Tate module. We devote an extensive section to the recollection of the theory developed in [loc. cit], which generalizes for the greatest part to the case of a base of higher dimension without any difficulty. We also introduce the purity maps (\cref{subs:purity_map}), which are something peculiar to bases of dimension higher than $1$ and are central in the definition of toric additivity.
 
In \cref{section2}, we recall the definition of a N\'eron model (\cref{defn_NM}) and state a number of properties regarding the behaviour of N\'eron models under different sorts of base change. Then we pass to study the group of components of a N\'eron model in terms of the monodromy action, following \citep[IX, 11]{SGA7}.

\Cref{section3,section4} form the heart of the article; in \cref{section3}, we introduce the conditions of toric additivity (\cref{def_TA}), $l$-toric additivity (\cref{def_l_TA}), and weak toric additivity (\cref{def_w_TA}). Then we pass to analyzing the relation between the different notions, in \cref{weak_and_l} and \cref{A=B}. In the latter theorem, we show how to express toric additivity in terms of monodromy action on the $l$-adic Tate module, for primes $l$ invertible on the base. Then, we show that the different notions of toric additivity are equivalent in the case of jacobians of curves (\cref{subs:case_of_curves}), but not in the more general case of abelian schemes (\cref{example:uniformization}).

In \cref{section4}, we work under the assumption that the base $S$ is a $\mathbb Q$-scheme; the section is devoted to proving \cref{TA->Neron}. We  first introduce test-N\'eron models (\cref{def:test_NM}); we prove they are unique (\cref{test-NM_unique}) and then that, under the assumption that $\mathcal A/S$ is toric-additive, they also exist (\cref{main_thm}). After a result on descent of test-N\'eron models (\cref{weil}) along tamely ramified covers, we conclude the section by showing that test-N\'eron models are N\'eron models, again under the assumption of toric-additivity (\cref{test=Neron}).

Finally, in \cref{section5}, we prove theorem \ref{intro:converse}.

\section*{Acknowledgements}
A substantial part of this article is the product of my doctoral thesis. I would like to thank my supervisors David Holmes and Qing Liu for their guidance. A big thank goes to 
Hendrik W. Lenstra for pointing out a major error in an earlier version of the document, and to Bas Edixhoven for many useful discussions.

\section{Generalities on semiabelian models}\label{section1}
\subsection{Normal crossing divisors and transversal traits}
We work over a regular, locally noetherian, base scheme $S$.
\begin{definition} 
Given a regular, noetherian local ring $R$, a \textit{regular system of parameters} is a minimal subset $\{r_1,\ldots,r_d\}\subset R$ of generators for the maximal ideal $\m\subset R$.
\end{definition}
\begin{definition}
A \textit{strict normal crossing divisor} $D$ on $S$ is a closed subscheme $D\subset S$ such that, for every point $s\in S$, the preimage of $D$ in the local ring $\O_{S,s}$ is the zero locus of a product $r_1\cdot\ldots\cdot r_n$, where $\{r_1,\ldots,r_n\}$ is a subset of a regular system of parameters $\{r_1,\ldots,r_d\}$ of $\O_{S,s}$.
\end{definition}

Write $\{D_i\}_{i\in\mathcal I}$ for the set of irreducible components of $D$. Then each $D_i$, seen as a reduced closed subscheme of $S$, is regular and of codimension $1$ in $S$; moreover, for every finite subset $\mathcal J\subset \mathcal I$, the intersection $\bigcap_{j\in\mathcal J}D_j$ is regular, and each of its irreducible components has codimension $|\mathcal J|$.

\begin{definition}
A \textit{normal crossing divisor D} on $S$ is a closed subscheme $D\subset S$ for which there exists an \'etale surjective morphism $S'\ra S$ such that the base change $D\times_SS'$ is a strict normal crossing divisor on $S'$.
\end{definition}

Notice that for every geometric point $s$ of $S$, the preimage of a normal crossing divisor $D$ in the spectrum of the strict henselization $\O^{sh}_{S,s}$ is a strict normal crossing divisor.

\begin{definition}\label{trait_transversal}
A \textit{strictly henselian trait} $Z$ is an affine scheme with $\Gamma(Z,\O_Z)$ a strictly henselian discrete valuation ring. Suppose we are given a morphism $f\colon Z\ra S$ and a normal crossing divisor $D$ on $S$. Let $S^{sh}\to S$ be the strict henselization at the closed point of $Z$, $D'=D\times_SS^{sh}$ and $\tilde f\colon Z\to S^{sh}$ the induced morphism. We say that $f$ is \textit{transversal to $D$} if for every component $D_i$ of $D'$ seen with reduced scheme structure, $D'_i\times_{S^{sh}}Z$ is a reduced point or is empty. 
\end{definition}

\subsection{The tame fundamental group}
Suppose that $S$ is strictly henselian, with fraction field $K$, and write $D$ as the union of its irreducible components, $D=D_1\cup D_2\cup \ldots \cup D_n$. Let $U=S\setminus D$, and for each $i=1,\ldots,n$, $U_i=S\setminus D_i$. It is a consequence of Abhyankar's Lemma (\cite[XIII, 5.2]{SGA1}) that every finite etale morphism $V\ra U$, tamely ramified over $D$ (\citep[XIII, 3.2.c)]{SGA1}), with $V$ connected, is dominated by a finite \'etale $W/U$ given by $$\O(W)=\frac{\O(U)[T_1,\ldots,T_n]}{T_1^{m_1}-r_1,\ldots,T_n^{m_1}-r_n}$$
where the integers $m_1,\ldots,m_n$ are coprime to $p$. Denoting by $\mu_{r,U}$ the group-scheme of $r$-roots of unity, it follows that $\underline{\Aut}_U(W)=\prod_{i=1}^n\mu_{m_i,U}$. Then, the \textit{tame fundamental group} of $U$ is $$\pi_1^t(U)=\prod_{i=1}^n \pi_1^t(U_i)=\prod_{i=1}^n \widehat\Z'(1)$$
Here  $\widehat\Z'(1)=\prod_{l\neq p}\Z_l(1)$ and $\Z_l(1)=\lim \mu_{l^r}(\overline K)$ is non-canonically isomorphic to $\Z_l$, an isomorphism being given by a choice of a compatible system $(z_{l^r})_{r\geq 1}$ of primitive $l^r$-roots of unity in $\overline K$. 

For a prime $l\neq p$, the factor $\prod_{i=1}^n \Z_l(1)$ of $\pi_1^t(U))$ is the biggest pro-$l$ quotient of $\pi_1^t(U)$ and will be denoted by $\pi_1^{t,l}(U)$. It is the automorphism group of the fibre functor of finite \'etale morphisms $V\ra U$ of degree a power of $l$.

\subsection{Semiabelian schemes}
\begin{definition}
Let $\kappa$ be a field and $G/\kappa$ a smooth, commutative $\kappa$-group scheme of finite type. We say that $G/\kappa$ is \textit{semiabelian} if it fits into an exact sequence of fppf-sheaves over $\kappa$

\begin{equation}0\ra T \ra G \ra B \ra 0
\label{ex_seq_semiabelian}
\end{equation}

where $T/\kappa$ is a torus and $B/\kappa$ an abelian variety. In fact, the image of $T$ in $G$ is necessarily the maximal subtorus of $G$, so every semiabelian scheme admits a canonical exact sequence \cref{ex_seq_semiabelian}. We call $\mu:=\dim T$ the \textit{toric rank} of $G$ and $\alpha:=\dim B$ its \textit{abelian rank}. These numbers are stable under base field extensions. Notice that $G$ is automatically geometrically connected.

For a general base scheme $S$, a smooth commutative $S$-group scheme $\mathcal G/S$ of finite type is \textit{semiabelian} if for all points $s\in S$, the fibre $\mathcal G_s$ is semiabelian.
\end{definition}

Given a semiabelian scheme $\mathcal G/S$, we define the  functions \begin{equation} \label{mu_function}
\mu\colon S\ra \mathbb Z_{\geq 0},\;\;\; \alpha\colon S\ra \mathbb Z_{\geq 0}
\end{equation} which associate to a point $s\in S$ respectively the toric and abelian rank of $\mathcal G_s$. The sum $\mu+\alpha$ is the locally constant function with value the relative dimension of $\mathcal G/S$.

\subsection{The $l$-adic Tate module}
Let $S$ be a regular, strictly local scheme, with closed point $s$ and residue field $k=k(s)$ of characteristic $p\geq 0$. We write $K$ for the fraction field of $S$ and we fix a separable closure $K^s$. Let $D=D_1\cup\ldots\cup D_n$ be a normal crossing divisor on $S$ and $U=S\setminus D$. We are also given:
\begin{itemize}
\item an abelian scheme $A/U$ of relative dimension $d\geq 0$;
\item a smooth, separated $S$-group scheme of finite presentation $\mathcal A/S$, together with an isomorphism $\mathcal A\times_SU\ra A$, such that the fibrewise-connected component of identity $\mathcal A^0/S$ is semiabelian.
\end{itemize}

Let $l$ be a prime and $r\geq 0$ an integer; we denote by $\mathcal A[l^r]$ the kernel of the multiplication map $$l^r\colon \mathcal A\ra \mathcal A.$$ It is a closed subgroup scheme of $\mathcal A$, flat and quasi-finite over $S$. Its restriction $\mathcal A[l^r]_U/U$ is a finite $U$-group scheme of order $l^{2rd}$. 

We write $T_l\mathcal A$ for the inverse system of fppf-sheaves of $\Z/l^r\Z$-modules, $\{\mathcal A[l^r]\}_{r\geq 0}$, with transition morphisms $\mathcal A[l^r]\to \mathcal A[l^{r'}]$ given by multiplication by $l^{r-r'}$. 
The base changes via $\Spec K\to S$ and $\Spec k\to S$ give rise to two inverse systems $T_l\mathcal A_K$ and $T_l\mathcal A_k$, which can be seen as $l$-divisible groups, and, when $l\neq p$, as lisse $l$-adic sheaves. 

Let now $l\neq p$. In this case $\mathcal A[l^r]$ is also \'etale over $S$, finite over $U$, and tamely ramified over $D$. It follows that the action of $\Gal(K^s|K)$ on $A[l^r](K^s)$ factors via the quotient map
$$\Gal(K^s|K)\ra \pi_1^{t}(U)=\widehat{\Z}'(1)^n.$$ 
We write $G$ for $\pi_1^t(U)$ and $I_i$ for the $i$-th copy of $\widehat{\Z}'(1)$, so that $G=\bigoplus_{i=1}^n I_i$.

The datum of the lisse $l$-adic sheaf $T_l\mathcal A_K$ is equivalent to the datum of the \textit{$l$-adic Tate module}
$$T_lA(K^s)=\lim_r A[l^r](K^s),$$
together with the continuous action of $\pi_1^t(U)$. As an abelian group, $T_lA(K^s)$ is a free $\Z_l$-module of rank $2d$.

Now, we return to the situation where $l$ is any prime number. Over the closed point $s\in S$ there is a canonical exact sequence
$$0\ra T \ra \mathcal A^0_s\ra B\ra 0$$ with $B$ abelian and $T$ a torus; multiplication by $l$ is an epimorphism on $\mathcal A^0_s$ and it follows that we have an exact sequence of $l$-divisible groups
\begin{equation}0\ra T_lT\ra T_l\mathcal A^0_s\ra T_lB\ra 0. \label{ex_seq}
\end{equation}
Write $\mu$ and $\alpha$ for $\mu(s)=\dim T$ and $\alpha(s)=\dim B$. Taking heights of the $l$-divisible groups in the exact sequence \eqref{ex_seq}, we get 
\begin{itemize}
\item $\htt T_lT=\mu(s)$,
\item $\htt T_lB=2\alpha(s)$,
\item $\htt T_l\mathcal A^0_s=\mu(s)+2\alpha(s)=2d-\mu(s)$.
\end{itemize}
When $l\neq p$, the heights of these $l$-adic sheaves can be interpreted as the rank as $\Z_l$-module of the respective groups of $k$-valued points. 

One immediate consequence of the above equatiosn is that the function $\rk T_l\mathcal A^0=2d-\mu \colon S\to \mathbb Z_{\geq 0}$ is lower semi-continuous; it follows that the toric rank function $\mu$ is upper semi-continuous.

The following lemma is particularly useful:
\begin{lemma}\label{TlA=TlA0}
Let $l$ be a prime different from $p$. The inclusion of $l$-adic sheaves $T_l\mathcal A^0\hookrightarrow T_l\mathcal A$ restricts to an equality over the closed point $s$; that is,
\begin{equation}\label{eq6}
(T_l\mathcal A)_s=(T_l\mathcal A^0)_s
\end{equation}
\end{lemma}
\begin{proof}
To prove  this, it is enough to check that $T_l\mathcal A_s(k)=T_l\mathcal A^0_s(k)$. If $(x_v)_v$ is an element of the left-hand side, each $x_v$ is a $l^v$-torsion element of $\mathcal A_s(k)$ infinitely divisible by $l$. Let $\Phi$ be the group of components of $\mathcal A_s$; it is a finite abelian group, by the assumption that $\mathcal A$ is of finite presentation. Let $\phi_v$ be the image of $x_v$ in $\Phi$; then $\phi_v$ belongs to the $l^v$-torsion subgroup of $\Phi$. Moreover $\phi_v$ is infinitely divisible by $l$; it follows that $\phi_v=0$, and that $x_v$ lies in $\mathcal A^0_s(k)$. 
\end{proof}
\subsubsection{The fixed part of the Tate module}

Let again $l$ be any prime, and consider the $l^r$-torsion subscheme $\mathcal A[l^r]/S$. As $S$ is henselian, there is a canonical decomposition
$$\mathcal A[l^r]=\mathcal A[l^r]^f\sqcup \mathcal A[l^r]'$$
where $\mathcal A[l^r]^f/S$, called \textit{the fixed part} of $\mathcal A[l^r]$, is finite over $S$ and $\mathcal A[l^r]'_s=\emptyset$. It can be shown that $\mathcal A[l^r]^f$ is an $S$-flat subgroup-scheme of $\mathcal A[l^r]$, which is moreover \'etale if $l\neq p$. 

We define the \textit{fixed part} of $T_l\mathcal A$ as the inverse system $\{\mathcal A[l^r]^f\}_{r\geq 0}$, which is an $l$-divisible group. Of course we have an inclusion
$$T_l\mathcal A^f\into T_l\mathcal A$$
restricting to an isomorphism over the closed fibre.

\subsubsection{The fixed part for $l\neq p$}If $l\neq p$, by strict henselianity of $S$, each finite \'etale scheme $\mathcal A[l^r]^f$ is a disjoint union of copies of $S$ and the $l$-adic sheaf $T_l\mathcal A^f$ is constant. We find \begin{equation}\label{eq4}\mathcal A[l^r]^f(K^s)=\mathcal A[l^r]^f(K)=\mathcal A[l^r]^f(S)=\mathcal A[l^r]^f_s(k)=\mathcal A[l^r]_s(k).
\end{equation}
We will be interested in the $\Z_l$-module
$$T_lA(K^s)^f:=T_l\mathcal A^f(K^s)\subseteq T_lA(K^s)$$ 
which we call again the \textit{fixed part} of the $l$-adic Tate module. 

By taking the limit in \eqref{eq4} and applying \cref{TlA=TlA0}, we find \begin{equation} \label{eq5}T_lA(K^s)^f=T_l\mathcal A^f(S)=T_l\mathcal A_s(k)=T_l\mathcal A^0_s(k).
\end{equation}

This last equality enables us to determine the rank of the fixed submodule of the Tate module,
\begin{equation}\label{rk_fixed}
\rk T_lA(K^s)^f=2d-\mu=\rk T_lA(K^s)-\mu \end{equation}


Moreover, we have that \begin{equation}\label{eq7}  
T_l\mathcal A^0_s(k)\otimes_{\Z_l}\Z/l^r\Z=\mathcal A^0_s[l^r](k)\end{equation} since $\mathcal A^0_s(k)$ is $l$-divisible. 
Hence,
\begin{equation}\label{Tate0
}T_lA(K^s)^f\otimes_{\Z_l}\Z/l^r\Z=\mathcal A^0_s[l^r](k).
\end{equation}
In other words, $T_lA(K^s)^f\otimes_{\Z_l}\Z/l^r\Z$ is the submodule of $A[l^r](K^s)$ consisting of those points that extend to sections of the fibrewise-connected component of identity $\mathcal A^0$.

The following lemma gives us an alternative interpretation of the fixed part of $T_lA(K^s)$:

\begin{lemma}\label{fixedpartG}
Let $l\neq p$. The submodule $T_lA(K^s)^f$ is the submodule $T_lA(K^s)^G\subseteq T_lA(K^s)$ of elements fixed by $G=\pi_1^t(U)$.
\end{lemma}
\begin{proof}
We treat first the case $\dim S=1$; so $S$ is the spectrum of a discrete valuation ring. In this case, $A/K$ admits a N\'eron model, $\mathcal N/S$. By assumption, the fibrewise-connected component of identity $\mathcal A^0$ is semiabelian, and we have an identification $\mathcal N^0=\mathcal A^0$ (\cref{A=N0}).

Now, equality \eqref{eq5} and \cref{TlA=TlA0} tell us that 
$$T_lA(K^s)^f=T_l\mathcal A^0_s(k)=\mathcal T_l\mathcal N^0_s(k)=T_l\mathcal N_s(k).$$
By Hensel's lemma, $\mathcal N_s[l^r](k)=\mathcal N[l^r](S)$ and by the definition of N\'eron model the latter is equal to $\mathcal N_K[l^r](K)=A[l^r](K^s)^G$. 
Hence, $T_lA(K^s)^G=\lim A[l^r](K^s)^G$ is equal to $T_l\mathcal N_s(k)=T_lA(K^s)^f$ and we are done.

Let now $S$ have dimension $\dim S\geq 2$. First, observe that $T_lA(K^s)^f\subseteq T_lA(K^s)^G$: indeed, as $T_l\mathcal A^f$ is constant, its $K^s$-valued point are actually $K$-valued. 

We show the reverse inclusion. We start by claiming that there exists a closed subscheme $Z\subset S$, regular and of dimension $1$, such that $Z\not\subseteq D$. For this, let $\{t_1,\ldots,t_n\}$ be a system of regular parameters of $\O(S)$, cutting out the divisor $D$. We complete the above set to a maximal system $\{t_1,\ldots,t_n,t_{n+1},\ldots,t_{\dim S}\}$ of regular parameters and let $Z=Z(t_1-t_2,t_2-t_3,\ldots,t_{n-1}-t_n,t_{n+1},t_{n+2},\ldots,t_{dim S})$. Now, $\O(Z)$ is a strictly henselian discrete valuation ring, and the generic point $\zeta$ of $Z$ lies in $U$. 

We let $L=k(\zeta)$ and $H=\Gal(L^s|L)$ for some separable closure $L\hookrightarrow L^s$. Since $\mathcal A[l^r]$ is finite \'etale over $U$, we have $\mathcal A[l^r](K)\subseteq \mathcal A[l^r](L)$ and by passing to the limit we obtain $T_lA(K^s)^G\subseteq T_lA(L^s)^H$. Moreover, by the dimension $1$ case, $T_lA(L^s)^H=T_l(\mathcal A_Z)(L^s)^f=T_l\mathcal A_s(k)$; the latter is equal to $T_lA(K^s)^f$, concluding the proof.
\end{proof}

\subsubsection{The toric part of the Tate module}\label{subsection_toric_part}
Denote by $\mathcal T_s$ the biggest subtorus of the semiabelian scheme $\mathcal A^0_s$ and let for the moment $l\neq p$. We have an inclusion of the $l^r$-torsion
$$\mathcal T_s[l^r]\subseteq \mathcal A^0_s[l^r].$$
As the restriction functor between the  category of finite \'etale $S$-schemes and the category of finite \'etale $k$-schemes is an equivalence of categories, we obtain a canonical finite \'etale $S$-subscheme of $\mathcal A^0[l^r]^f$,  called the \textit{toric part} of $\mathcal A^0[l^r]$,
$$\mathcal A^0[l^r]^t\hookrightarrow \mathcal A^0[l^r]^f\hookrightarrow \mathcal A^0[l^r]$$
such that $A^0[l^r]^t\otimes_S k=\mathcal T_s[l^r]$.

Taking the limit, we find a constant $l$-adic subsheaf $T_l\mathcal A^t$ of $T_l\mathcal A^f.$
Then, passing to the generic fibre, we obtain a submodule $T_lA(K^s)^t$ of $T_lA(K^s)^f=T_lA(K^s)^G\subseteq T_lA(K^s)$, which we call \textit{toric part} of $T_lA(K^s)$. Its rank is of course the rank of the $\Z_l$-module $T_l\mathcal T_s(k)$, that is
\begin{equation}\label{rk_toric}
\rk T_lA(K^s)^t=\mu.\end{equation} 

To summarize, we have a filtration of the $l$-adic Tate module
$$0\xhookrightarrow{\mu} T_lA(K^s)^t\xhookrightarrow{2\alpha} T_lA(K^s)^f\xhookrightarrow{\mu} T_lA(K^s)$$
where the numbers on top of the arrows are the ranks of the successive quotients in the filtration.

If $l=p$, it is still possible to construct a canonical $p$-divisible group $T_p\mathcal A^t\into T_p\mathcal A^f$, which, when restricted to the closed fibre, gives the inclusion of $p$-divisible groups $T_p\mathcal T_s\into T_p\mathcal A_s$. As explained in \citep{SGA7}[IX, 5.1. and 6.1.], the functor $Sch/S\to Sets$ of finite subgroup-schemes of $\mathcal A/S$ of multiplicative type, is representable by an \'etale scheme of finite type over $S$. Hence, for every $r$, there is a unique subgroup scheme $\mathcal A[p^r]^t\subset \mathcal A$ restricting to $\mathcal T[p^r]_s$. The $p$-divisible group $T_p\mathcal A^t$ is given by the system $\{\mathcal A[p^r]^t\}_{r\geq 0}.$ 

\subsubsection{The dual abelian variety and the Weil pairing}
We will now only focus on the semi-abelian scheme $\mathcal A^0\subset \mathcal A$; for this reason, we will write simply $\mathcal A$ for it, rather than $\mathcal A^0$. Consider the dual abelian variety $A'_K$ of $A_K$. By \citep[IV, 7.1]{pinceaux}, there exists a unique semi-abelian scheme $\mathcal A'/S$ extending $A'_K$. 
Let $\phi\colon A_K\ra A'_K$ be an isogeny; it extends uniquely to a group homomorphism $\mathcal A\ra \mathcal A'$ (\citep{faltings1990degeneration}, I, 2.7) inducing isogenies
$$\mathcal T_s\ra \mathcal T'_s, \;\;\; \mathcal B_s\ra \mathcal B'_s$$
between the toric and abelian parts of $\mathcal A_s$ and $\mathcal A'_s$. 
We deduce the equality between the toric and abelian ranks 
$$\mu=\mu'\;\;\; \alpha=\alpha'.$$

By \citep[II, 3.6]{pinceaux} the natural functor 
$$BIEXT(\mathcal A,\mathcal A'; \mathbb G_{m,S})\ra BIEXT(\mathcal A_K,\mathcal A'_K; \mathbb G_{m,K})$$
is an equivalence of categories; thus the Poincar\'e biextension on $A_K\times_KA'_K$ extends uniquely to a biextension on $\mathcal A\times_S\mathcal A'$, and we obtain for every prime $l$ a perfect pairing
\begin{equation} \label{pairing}
T_l\mathcal A\times T_l\mathcal A'\ra T_l(\mathbb G_m)=\underline{\Z}_l(1)
\end{equation}
of projective systems on $S$ extending the Weil pairing $$\chi\colon T_lA_K\times T_lA'_K\ra \underline \Z_l(1).$$

\begin{theorem}[Orthogonality theorem]\label{ort_thm}
The toric part $T_l\mathcal A^t_K$ is the orthogonal of the fixed part $T_lA'^f_K$ via the Weil pairing $\chi$.
\end{theorem}
\begin{proof}
See \citep[IX, 5.2]{SGA7}. The proof for $\dim S>1$ is the same, using the pairing \eqref{pairing}.
\end{proof}
\subsubsection{The action of $G$ on the Tate module is unipotent}\label{subsub_unipotent}
The orthogonality \cref{ort_thm} permits to describe more explicitly the Galois action on the $l$-adic Tate module when $l\neq p$.
\begin{lemma}
There exists a submodule $V$ of $T=T_lA(K^s)$ such that $G=\pi_1^t(U)$ acts trivially on $V$ and on the quotient $T/V$. 
\end{lemma}
\begin{proof}
Clearly, $G$ acts trivially on $V=T^G=T_lA(K^s)^f$. Now, as $T^G$ is orthogonal to $T'^t$ (where $T'=T_lA'(K^s)$) via the pairing $\chi$, we obtain a perfect pairing $T/T^G\times T'^t\ra \Z_l(1)$ which identifies $T/T^G$ with $\Hom_{\Z_l}(T'^t,\Z_l(1))$. As $T'^t\subset T'^G$, we conclude that $G$ acts trivially on $T/T^G$. 
\end{proof}

It follows from the above proposition that the action of $G$ on $T_lA(K^s)$ is unipotent of level $2$: that is, writing 
$$\rho\colon G\ra \Aut(T_lA(K^s)\otimes_{\Z_l}\Q_l),$$
we have for every $g\in G$
$$(\rho(g)-\id)^2=0.$$

Because the profinite group $G$ acts on a $\Q_l$-vector space unipotently and continuously, the image of $\rho$ is a pro-$l$-group. Thus, the action of $G$ factors via its biggest pro-$l$-quotient 
$$G=\widehat{\Z}(1)^n=\pi_1^t(U)\ra \pi_1^{t,l}(U)=\Z_l(1)^n.$$
\subsection{Character groups}\label{subs:char_grps}
Consider in addition to the semiabelian scheme $\mathcal A$ also $\mathcal A'$, the unique semiabelian prolongation of the dual abelian scheme $\mathcal A'_U$. Let $T,T'$ be the biggest subtori of the fibres $\mathcal A_s$, $\mathcal A'_s$, defined over the residue field $k=k(s)$ at the closed point $s\in S$, and $X=\Hom_k(T,\mathbb G_{m,s}), X'=\Hom_k(T',\mathbb G_{m,s})$ be their character groups, which are constant sheaves of free abelian groups of the same rank $\mu=\dim T=\dim T'$. In fact, $X$ and $X'$ are respectively determined by the groups $X(k)$, $X'(k)$, so we will confuse them with their groups of $k$-valued points.

The sheaves $X,X'$ lift uniquely to $\mathcal X, \mathcal X'$, constant sheaves over $S$. These can be related to the $l$-divisible group $T_l\mathcal A^t$ of \cref{subsection_toric_part}. Indeed for all primes $l$ we have a natural isomorphism $T[l^r]_s=X^{\vee}\otimes_{\Z/l^r\Z}\mu_{l^r,s}$, where the dual sign stands for dual with values in $\Z$. The two sides of the isomorphism lift uniquely to finite group schemes of multiplicative type over $S$, and we get a natural isomorphism
$$\mathcal A[l^r]^t=\mathcal X^{\vee}\otimes \mu_{l^r,S}.$$

Taking the limit over $r$, we obtain canonical isomorphisms of $l$-divisible groups
\begin{equation}\label{relation_toric_char}\begin{cases}
T_l\mathcal A^t &=\mathcal X^{\vee}\otimes\Z_l(1) \\
T_l\mathcal A'^t &=\mathcal X'^{\vee}\otimes\Z_l(1). 
\end{cases}\end{equation}

Passing to the generic fibre and using the orthogonality theorem \ref{ort_thm}, it follows that there are also canonical isomorphisms
\begin{equation}\label{relation_fixed_char}\begin{cases}
T_l\mathcal A_K/T_l\mathcal A_K^f &=\mathcal X'_K\otimes\Z_l \\
T_l\mathcal A_K'^t/T_l\mathcal A_K'^f &=\mathcal X_K\otimes\Z_l
\end{cases}
\end{equation}

\subsection{Relations between maximal tori of different fibres}\label{sec:relation_tori}
We state first the following fact:
\begin{lemma}\label{lemma:tori_are_split}
Let $t$ be any point of $S$. The maximal torus $T$ contained in the fibre $\mathcal A_t$ is split. 
\end{lemma}
\begin{proof}
Let $Z\subset S$ be the schematic closure of $t$ in $S$. It is itself a strictly local scheme. Consider the schematic closure $\mathcal T\subset \mathcal A_Z$ of $T$, which is a closed subgroup scheme of $\mathcal A_Z$, and hence a subtorus (by looking at torsion points). Now, $\mathcal T$ is completely determined by its locally constant sheaf of characters $\mathcal X=\Hom_Z(\mathcal T,\mathbb G_{m,Z})$. Since $Z$ is strictly henselian, $\mathcal X/S$ is constant, i.e., $\mathcal T$ is split, and so is $T$.
\end{proof}

Because of the previous lemma, whenever we work over a strictly henselian local base, we may confuse groups of characters of subtori of fibres of $\mathcal A/S$ with their underlying abstract free abelian groups. 

Now let $x$ be a point of $S$ and let $Z\subset S$ be its schematic closure. Let $T^{(x)}$ be the maximal torus over the fibre $\mathcal A_x$, with character group $X^{(x)}$. The schematic closure of $T^{(x)}$ in $\mathcal A$ is a split torus $\mathcal T$ over $Z$. By taking the fibre at the closed point $s$, we obtain an inclusion of tori $\mathcal T_s\subset T$, where $T$ is the maximal torus at the closed point $s$. This induces a surjective homomorphism of groups of characters 
\begin{equation}\label{def:specialization_map}sp_{x}\colon X\twoheadrightarrow X^{(x)}
\end{equation}
which we call \textit{specialization map} on group of characters.

The inclusion of tori $\mathcal T_s\subset T$ induced for all primes $l$ and $r\geq 0$ an inclusion of torsion $k$-schemes $ \mathcal T_s[l^r]\subset T[l^r]$. By the equivalence of categories between finite subgroup schemes of multiplicative type over $S$ and over $s$, we obtain a finite $S$-subscheme $$\mathcal A[l^r]^{t(x)}\subset \mathcal A[l^r]^t.$$

We call this the \textit{toric part} of $\mathcal A[l^r]$ at $x$. Taking the limit over $r$ we obtain a $l$-divisible subgroup
$$T_l\mathcal A^{t(x)}\subset T_l\mathcal A^t.$$

We define $T_l\mathcal A'^{f(x)}$ to be its orthogonal by the pairing $\chi$. In a similar way, we define $T_l\mathcal A'^{t(x)}$ an $T_l\mathcal A^{f(x)}$. 

Let $S^{x}\to S$ be a strict henselization at $x$. Because the scheme $\mathcal T/Z$ is a split torus, $\mathcal T[l^r]/Z$ is \textit{finite}, from which it follows that the restriction of $\mathcal A[l^r]^{t(x)}$ to $S^{x}$ coincide with the toric part of $\mathcal A[l^r]_{S^{x}}$ as defined in \cref{subsection_toric_part}. A similar statement holds for $T_l\mathcal A^{t(x)}$.

In particular, for $l\neq p$, we have for the generic fibre of $T_l\mathcal A^{f(x)}$ the equality
$$T_lA(K^s)^{f(x)}=T_lA(K^s)^{G^{x}}$$
where $G^{x}=\pi_1^{t,l}(S^{x}\times_SU)$ is identified with the quotient $G=\pi_1^{t,l}(U)\to \prod_{D_i\ni x} \pi_1^{t,l}(S\setminus D_i)=\prod_{D_i\ni x} \Z_l(1)$.

The natural morphisms
$$\frac{T_l\mathcal A_K}{T_l\mathcal A_K^f}\twoheadrightarrow \frac{T_l\mathcal A_K}{T_l\mathcal A_K^{f^{(x)}}}$$
and
$$T_l\mathcal A^{t(x)}\hookrightarrow T_l\mathcal A^t$$
coincide with the morphisms induced by the specialization maps 
$$X'\otimes\Z_l\twoheadrightarrow X'^{(x)}\otimes\Z_l$$
and 
$$\mathcal X^{(x)\vee}\otimes\underline \Z_l(1)\hookrightarrow \mathcal X^{\vee}\otimes \underline\Z_l(1)$$

\subsection{Purity maps}\label{subs:purity_map}
Let $D_1,\ldots,D_n$ be the components of the divisor $D$, with generic points $\zeta_1,\ldots,\zeta_n$.  

Given $t\in S$, and $1\leq i \leq n$ such that $t\in D_i$, we have a specialization map $X_t\twoheadrightarrow X_i$ from the maximal torus of $\mathcal A_t$ to the maximal torus $X_i$ of $\mathcal A_{\zeta_i}$. 

If we associate to $t$ a subset $J_t\subset \{1,\ldots,n\}$ such that $i\in J_t$ if and only if $t\in D_i$, then we obtain a natural homomorphism
\begin{equation}\label{purity_map}
p_t\colon X_t\to \bigoplus_{i\in J_t}X_i
\end{equation}

\begin{definition}\label{def_purity_map}
We call the homomorphism $p_t$ of \eqref{purity_map} the \textit{purity map} at $t$.
\end{definition}

Clearly, purity maps are compatible with specialization, that is, if $t$ specializes to a point $t'$, the following diagram commutes
\begin{center}
\begin{tikzcd}
X_{t'} \arrow[r, "p_{t'}"] \arrow[d, "sp_{t,t'}"] & \bigoplus_{i\in J_{t'}} X_i \arrow[d, "res"] \\
X_t \arrow[r, "p_t"] & \bigoplus_{i\in J_t} X_i 
\end{tikzcd}
\end{center}
where the map res is the projection induced by the inclusion $J_t\subset J_{t'}$. 

let $l$ be a prime different from $p$. Tensoring with $\Z_l$, the purity map becomes the map
\begin{equation}\label{map:purity_l}
p_{t,l}\colon X_{t}\otimes\Z_l=\frac{T_l\mathcal A'(K^s)}{T_l\mathcal A'(K^s)^{G_t}}\to \bigoplus_{i\in J_t} X_i\otimes\Z_l=\bigoplus_{i\in J_t} \frac{T_l\mathcal A'(K^s)}{T_l\mathcal A'(K^s)^{I_i}}
\end{equation}
where $G_t=\pi_1^{t,l}(S\setminus \bigcup_{i\in J_t}D_i)=\bigoplus_{i\in J_t}\pi_1^{t,l}(S\setminus D_i)=\bigoplus_{i\in J_t} I_i$, and where the map toward each factor is the natural projection. We immediately obtain the following:
can be read from
\begin{lemma}\label{lemma:purity_injective}
For any $t\in S$, the purity map $p_t$ is injective. 
\end{lemma}
\begin{proof}
As $X_t$ is a free abelian group, it suffices to check injectivity of $p_{t,l}$ for some prime $l$, which we pick different from $p$. Clearly, $p_{t,l}$ is injective, as $\bigcap T_lA'(K^s)^{I_i}= T_lA'(K^s)^{G_t}$. 
\end{proof}

So, in a sense, all information on character groups of the maximal tori of the fibre is contained in the groups $X_i$ for $i=1,\ldots,n$ (hence the name \textit{purity map} for the map of \cref{purity_map}).

\begin{corollary}\label{coro:mu_relation}
The function $\mu\colon S\to \Z_{\geq 0}$ satisfies the inequality
\begin{equation}\label{mu_relation}
\mu(t)\leq \sum_{i\in J_t}\mu(\zeta_i).
\end{equation}
\end{corollary}
\begin{proof}
It follows by taking ranks in the purity map, which is injective by \cref{lemma:purity_injective}.
\end{proof}

\Cref{coro:mu_relation} and the upper semi-continuity of $\mu$ tell us that $\mu\colon S\to \Z_{\geq 0}$ is actually constant on each locally closed piece of the natural stratification of the divisor $D$.

\subsection{The monodromy pairing}\label{subsubsection_monodromy_pairing}
Assume now that the base $S$ is of dimension $1$; that is, let $S$ be the spectrum of a strictly henselian discrete valuation ring. In \citep[IX]{SGA7} the authors define the \textit{monodromy pairing}, from which one can deduce a particularly nice description of the component group of the N\'eron model of $\mathcal A_K$ over $S$ (see \cref{group_comps}). We are going to follow \citep{SGA7} and briefly present the subject.

Let $l\neq p$ be a prime. We recall that we have an action of $\Gal(K^{s}|K)$ on the $\Z_l$-module $T_lA(K^{s})$ and $T_lA^f(K^{s})$ is the submodule of Galois-fixed elements. Moreover, the action factors via the maximal pro-$l$ quotient $\Gal(K^s|K)\rightarrow G\cong \Z_l(1)$. 

Consider the Weil pairing $T_lA(K^{s})\times T_lA'(K^{s})\to \Z_l(1)$. The pairing is compatible with the Galois action (where $G$ acts trivially on $\Z_l(1)$); so if we take $g\in G$, $x\in T_lA(K^{s})$ and $y\in T_lA'(K^{s})^G$, we obtain
$$(gx-x,y)=\frac{(gx,y)}{(x,y)}=\frac{(x,g^{-1}y)}{(x,y)}=\frac{(x,y)}{(x,y)}=1.$$
By the theorem of orthogonality, $T_lA(K^{s})^t$ is the orthogonal of $T_lA'(K^{s})^G$; hence $gx-x\in T_lA(K^{s})^t$. Notice that $gx-x=0$ for $x\in T_lA(K^{s})^G$; thus we have found a map
\begin{align}\label{map_for_monodromy}G=\Z_l(1) &\to \Hom_{\Z_l}(T_lA(K^{s})/T_lA(K^{s})^G, T_lA(K^{s})^t) \\
g &\mapsto (x \mapsto gx-x) \nonumber
\end{align}

Relations \cref{relation_toric_char} and \cref{relation_fixed_char} give us $T_lA(K^s)^t=X^{\vee}\otimes \Z_l(1)$ and $T_lA(K^s)/T_lA(K^s)^G=X'\otimes \Z_l.$. Now, for a $\Z$-module $M$, we write shorthand $M_l:=M\otimes\Z_l$. The map in \cref{map_for_monodromy} can be rewritten as 
$$\Z_l(1)\to \Hom_{\Z_l}(X'_l,X^{\vee}_l\otimes \Z_l(1))=X'^{\vee}_l\otimes X^{\vee}_l\otimes \Z_l(1)$$

which gives a canonical element of $X'^{\vee}_l\otimes X^{\vee}_l$, or in other words a bilinear pairing \begin{equation}\label{monodromy_for_l}\phi_l\colon X_l\otimes X'_l\to \Z_l.
\end{equation}

A more concrete description of the pairing \eqref{monodromy_for_l} can be given as follows: first, one chooses a topological generator $\sigma$ for $G=\Z_l(1)$, which amounts to fixing an isomorphism of $\Z_l$-modules $\Z_l\cong\Z_l(1)$. Then the homomorphism 
$$X'_l\to X^{\vee}_l\cong X^{\vee}_l\otimes\Z_l(1)$$ associated to $\phi_l$ is the homomorphism
\begin{align}\label{homomorphism_monodromy}
\frac{T_lA(K^s)}{T_lA(K^s)^G} &\to T_lA(K^s)^t\\
x &\mapsto \sigma x - x \nonumber
\end{align}
As the homomorphism \eqref{homomorphism_monodromy} is injective and between free $\Z_l$-modules of the same rank, it follows that the pairing in \cref{monodromy_for_l} is non-degenerate.

\begin{theorem}[\cite{SGA7}, IX.10.4]\label{thm:monodromy_pairing}
There exists a unique non-degenerate, bilinear pairing 
\begin{equation}
\phi \colon X\otimes X'\to \Z
\end{equation} such that for every $l\neq p$ the restriction to $\Z_l$ gives the pairing in \cref{monodromy_for_l}. Moreover, a choice of polarization $\xi\colon \mathcal A'_K\to\mathcal A_K$ induces an isogeny $\xi\colon X\to X'$ and the induced pairing
\begin{align*}
X\otimes X &\to \Z \\
x\otimes y & \mapsto \phi(x\otimes \xi(y))
\end{align*}
is symmetric and positive definite.
\end{theorem}

\subsection{Monodromy pairing and transversal traits}
We return to hypotheses and notations of \cref{subs:purity_map}. Let $Z$ be a strictly henselian trait with closed point $z$ and fraction field $L=\Frac \O(Z)$, and write $\nu\colon L^{\times}\to \Z$ for the natural valuation. Let $f\colon Z\to S$ be a morphism sending the generic point of $Z$ inside $U=S\setminus D$. To each irreducible component $D_i$ of $D$, we can associate an element $u_i\in \mathcal O(S)$ defining $D_i$. The non-negative number $a_i=\nu(f^*(u_i))$ does not depend on the choice of $u_i$. We obtain an $n$-uple $(a_1,\ldots,a_n)\in \Z_{\geq 0}^n$. For $j\in \{1,\ldots,n\}$, the condition $a_j>0$ is equivalent to $f(z)\in D_j$. Moreover, the morphism $f$ is transversal (\cref{trait_transversal}) if and only if every $a_i$ is either $1$ or $0$.

Now let $X_1,\ldots,X_n$ (resp. $X_1',\ldots,X_n'$) be the groups of characters of the maximal tori of $\mathcal A$ (resp. $\mathcal A'$) at the generic points of $D_1,\ldots,D_n$. There are $n$ monodromy pairings, $\phi_1,\ldots,\phi_n$, with $\phi_i\colon X'_i\to X^{\vee}_i$. 

Let $J\subset\{1,\ldots,n\}$ be the subset of those $j$ with $a_j>0$. Then we obtain one more monodromy pairing,
$$\phi_f\colon Y'\to Y^{\vee}$$ 
where $Y$ and $Y'$ are the character groups of $\mathcal A$ and $\mathcal A'$ at the generic point of $\cap_{j\in J} D_j$. Notice that the purity maps give us an injective morphism 
$$p'\colon Y'\to \bigoplus_{j\in J}X'_j$$
and a surjective one 
$$p^{\vee}\colon \bigoplus X_j^{\vee}\to Y^{\vee}.$$

\begin{proposition}\label{prop:monodromy_transversal}
The monodromy pairing $\phi_f$ is the composition
$$F\colon Y'\xrightarrow{p'} \bigoplus_{j\in J} X'_j\xrightarrow{(a_j\phi_j)_{j\in J}}\bigoplus_{j\in J} X^{\vee}_j\xrightarrow{p^{\vee}} Y^{\vee}.$$
In particular, $\phi_f$ depends only on the tuple $(a_j)_{j\in J}$.
\end{proposition}

\begin{proof}
Let $l\neq p$ be a prime. We denote by a $\wt{\cdot}$ objects tensored with $\Z_l$. It suffices to prove the statement after tensoring with $\Z_l$; that is, to check that $\wt{F}=\wt{\phi_f}$. In fact, it suffices to check that $i\circ\wt{F}=i\circ\wt{\phi_f}$, where $i\colon Y^{\vee}\into T_lA(K^s)$ is the natural inclusion.

The morphism $Z\to S$ induces a group homomorphism of tame pro-$l$ fundamental groups
$$H:=\pi_1^{t,l}(Z\setminus\{z\})=\Z_l(1)\to \pi_1^{t,l}(S\setminus \bigcup_J D_j)=\bigoplus_J I_j$$ with $I_j=\pi_1^{t,l}(S\setminus D_j)\cong \Z_l(1),$
which sends a topological generator $\sigma$ of $H$ to a sum $\sum_{j\in J}a_j\sigma_j$, where $\sigma_j$ is a topological generator of $I_j$.

The above fact expresses the commutativity of the diagram
\begin{center}
\begin{tikzcd}
Y'\otimes\Z_l=\frac{T_lA_L(L^{sep})}{T_lA_L(L^{sep})^{H}} \arrow[rr, "\sigma-1"]\arrow[d] && T_lA_L^t(L^{sep})=Y^{\vee}\otimes\Z_l\\
\bigoplus(X'_j\otimes\Z_l)=\bigoplus \frac{T_lA_K(K^{sep})}{T_lA_K(K^{sep})^{I_j}}\arrow[rr, "(a_j\sigma_j-1)_{j\in J}"]&& \bigoplus T_lA_K^{t_j}(K^{sep})=\bigoplus (X_j^{\vee}\otimes\Z_l) \ar[u]
\end{tikzcd}
\end{center}

The upper map is $\phi_f\otimes \Z_l$; the lower one is $((a_j\phi_j)\otimes\Z_l)_{j\in J}$; the vertical maps are $p'$ and $p^{\vee}$. This concludes the proof.

\end{proof}

\section{N\'eron models and their groups of components}\label{section2}
\subsection{N\'eron models of abelian schemes}
Let $S$ be any scheme, $U\subset S$ an open and $A/U$ an abelian scheme. 

\begin{definition}\label{defn_NM}
A \textit{N\'eron model} for $A$ over $S$ is a smooth, separated algebraic space \footnote{defined as in \citep{stacks}\href{http://stacks.math.columbia.edu/tag/025Y}{TAG 025Y}.)} $\mathcal N/S$ of finite type, together with an isomorphism $\mathcal N\times_SU\rightarrow A$, satisfying the following universal property: for every smooth morphism of schemes $T\ra S$ and $U$-morphism $f\colon T_U\ra A$, there exists a unique morphism $g\colon T\ra \mathcal N$ such that $g_{|U}=f.$
\end{definition}

Here are two easy-to-check fundamental properties of N\'eron models:
\begin{itemize}
\item[i)] Given two N\'eron models $\mathcal M,\mathcal N$ over $S$ for $A$ there is a unique $S$-isomorphism $\mathcal M\to \mathcal N$ compatible with the identifications of $\mathcal M_U$ and $\mathcal N_U$ with $A$.
\item[ii)]applying the defining universal property of N\'eron models to the morphisms $m\colon A\times_U A\ra A, i\colon A\ra A,$ and $0_A\colon U\ra A$ defining the group structure of $A$, we see that $\mathcal N/S$ inherits from $A$ a unique $S$-group-space structure.
\end{itemize}

We also introduce a similar object, which satisfies a weaker universal property:
\begin{definition}\label{def_weak_NM}
A \textit{weak N\'eron model} for $A$ over $S$ is a smooth, separated algebraic space $\mathcal N/S$ of finite type, together with an isomorphism $\mathcal N\times_SU\ra A$, satisfying the following universal property: every section $U\ra A$ extends uniquely to a section $S\ra \mathcal N$.
\end{definition}

In particular, a N\'eron model is a weak N\'eron model. Notice that in the case of weak N\'eron models, we do not have any uniqueness statement, and they need not necessarily inherit a group structure from $A$.

We point out that our \cref{def_weak_NM} of weak N\'eron model differs slightly from the one normally found in the literature: the latter requires that the universal property is satisfied for all \'etale quasi-sections.

\subsection{Base change properties}
We list some properties of compatibility of formation of N\'eron models with base change. The same properties are to be found in \citep[3.2]{Orecchia} with their proofs. Here we report only the statements.

 In general, the property of being a N\'eron model is not stable under arbitrary base change. However, we have that:
\begin{lemma}\label{smooth_base_change}
Let $\mathcal N/S$ be a N\'eron model of $A/U$; let $S'\ra S$ be a smooth morphism and $U'=U\times_SS'$. Then the base change $\mathcal N\times_SS'$ is a N\'eron model of $A_{U'}$.
\end{lemma}

\begin{lemma}\label{fpqclocal}
Let $\mathcal N/S$ be a smooth, separated algebraic space of finite type with an isomorphism $\mathcal N\times_SU\ra A$. Let $S'\ra S$ be a faithfully flat morphism and write $U'=U\times_SS'$. If $\mathcal N\times_SS'$ is a N\'eron model of $A\times_UU'$, then $\mathcal N/S$ is a N\'eron model of $A$.
\end{lemma}

\begin{lemma}\label{desc_smooth}
Let $A/U$ be abelian, $f\colon S'\ra S$ a smooth surjective morphism, $U'=U\times_SS'$, and $\mathcal N'/S'$ a N\'eron model of $A\times_SS'$.
Then there exists a N\'eron model $\mathcal N/S$ for $A$.
\end{lemma}

\begin{lemma}\label{NMlocaliz}
Assume $S$ is locally noetherian. Let $s$ be a point (resp. geometric point) of $S$ and $\til S$ the spectrum of the localization (resp. strict henselization) at $s$. Suppose that $\mathcal N/S$ is a N\'eron model for $A/U$. Then $\mathcal N\times_S\til S$ is a N\'eron model for $A\times_U \til U$, where $\til U=\til S\times_S U$.
\end{lemma}

\begin{proposition}
Assume that $S$ is regular. If $\mathcal A/S$ is a an abelian algebraic space, then it is a N\'eron model of its restriction $\mathcal A\times_SU$.
\end{proposition}
Finally, here is the main theorem about existence of N\'eron models in the case where the base $S$ is of dimension $1$.

\begin{theorem}[\citep{BLR}, 1.4/3]\label{dim1}
Let $S$ be a connected Dedekind scheme with fraction field $K$ and let $A/K$ be an abelian variety. Then there exists a N\'eron model $\mathcal N$ over $S$ for $A/K$. 
\end{theorem}

\subsection{Connected components of a N\'eron model} \label{group_comps}
Assume now that $S$ is the spectrum of a regular, strictly henselian local ring, with closed point $s$, residue field $k=k(s)$. We make a choice $K\into K^s$ of separable closure of the fraction field. Let $U\subset S$ be the open complement of a normal crossing divisor $D=D_1\cup\ldots\cup D_n$; we write $G=\bigoplus_{i=1}^nI_i$ where $G=\pi_0^t(U)$ and $I_i=\pi_0^t(S\setminus D_i)\cong \widehat\Z'$. 

We let $A/U$ an abelian scheme extending to a semiabelian $\mathcal A/S$. We write $A'/U$ for the dual abelian scheme, and $\mathcal A'/S$ for its unique semiabelian extension. We assume that $A/U$ admits a N\'eron model $\mathcal N/S$.

\subsubsection{The fibrewise-connected component $\mathcal N^0$}
Thanks to the assumption that $A/U$ admits a semiabelian prolongation, the fibrewise-connected component of identity of $\mathcal N$ is known:

\begin{lemma} \label{A=N0}
Suppose $A/U$ admits a N\'eron model $\mathcal N/S$. Then the canonical morphism $\mathcal A\ra \mathcal N$ is an open immersion, and induces an isomorphism from $\mathcal A^0$ to the fibrewise-connected component of identity $\mathcal N^0$. In particular, $\mathcal N^0/S$ is semiabelian.
\end{lemma}
\begin{proof}
The fact that $\mathcal A\ra \mathcal N$ is an open immersion follows from \cite[IX, Prop. 3.1.e]{SGA7}. For every point $s\in S$ of codimension $1$, the restriction of $\mathcal N^0$ to the local ring $\O_{S,s}$ is the N\'eron model of its generic fibre, by \cref{NMlocaliz}. It follows by \citep[XI, 1.15]{raynaud} that the induced morphism $\mathcal A^0\ra \mathcal N^0$ is an isomorphism.
\end{proof}

\subsubsection{The group of components}
We review the explicit description of the group of connected components of a N\'eron model given in \citep[IX, 11]{SGA7} in terms of the monodromy pairing and the Tate module of $A$. 

We are interested in the \'etale $k$-group scheme of finite type $$\underline\Phi=\frac{\mathcal N_s}{\mathcal N^0_s}.$$

As $k$ is separably closed, $\underline\Phi$ is determined by its group of $k$-rational points $\Phi:=\underline\Phi(k)$, which is a finite abelian group. We have \begin{equation} \Phi=\frac{\mathcal N_s(k)}{\mathcal N^0_s(k)}= \frac{\mathcal N(S)}{\mathcal N^0(S)} =\frac{A(U)}{A(U)^0} \label{eq_phi} \end{equation}
where by $A(U)^0$ we denote the subset of $A(U)$ of $U$-points specializing to $S$-points of the identity component $\mathcal N^0$. The second equality is a consequence of Hensel's lemma and the third of the defining property of N\'eron models. 

\subsubsection{The prime-to-$p$ part}
We let $l$ be a prime different from the residue characteristic $p=$ char $k$ and $n\geq 0$ be an integer. Taking $l^n$-torsion in the exact sequence $$0\ra A(U)^0\ra A(U)\ra \Phi\ra 0$$ gives an exact sequence
$$0\ra A[l^n](U)^0\ra A[l^n](U)\ra \Phi[l^n]\ra A(U)^0/l^nA(U)^0$$
Multiplication by $l^n$ on the semiabelian scheme $\mathcal N^0_s$ is an \'etale and surjective morphism; it follows that $\mathcal N^0_s(k)=A(U)^0$ is $l$-divisible; hence
\begin{equation}\label{eq8}
\Phi[l^n]=\frac{A[l^n](U)}{A[l^n](U)^0}
\end{equation}

Writing $T$ for $T_lA(K^s)$, we see that $A[l^n](K)=(T\otimes\Z/l^n\Z)^G$. So we have an expression for the part of \cref{eq8} above the fraction line.

Let's turn to study $A[l^n](U)^0.$ The latter, by the defining property of N\'eron models, is simply $\mathcal N^0[l^n](S)$. Now, every section of a quasi-finite separated scheme over $S$ factors via its finite part, so $\mathcal N^0[l^n](S)=\mathcal N^0[l^n]^f(S)$. Because $\mathcal N^0(S)$ is $l$-divisible, we have $\mathcal N^0[l^n]^f(S)=(T_l\mathcal N^0)^f(S)\otimes_{\Z_l}\Z/l^n\Z$. Now, $(T_l\mathcal N^0)^f(S)$ is equal to $T_lA(K^s)^f$, which in turn is equal to $T_lA(K^s)^{G}$ by \cref{fixedpartG}. 

This shows that $A[l^n](K)^0=T_lA(K^s)^G\otimes_{\Z_l}\Z/l^n\Z$. We have found the relation

\begin{equation}\label{torsion_phi}
\Phi[l^n]=\frac{\left(T\otimes \Z/l^n\Z\right)^G}{T^G\otimes \Z/l^n\Z}.
\end{equation}

By taking the colimit over the powers of $l$ we find that the $l$-part $_l\Phi$ of the group of components is given by 
\begin{equation}\label{l_primary_phi}
_l\Phi=\colim _n\Phi[l^n]=\frac{\left(T\otimes \Q_l/\Z_l\right)^G}{T^G\otimes \Q_l/\Z_l}.
\end{equation}

Let now $X, X'$ be the character groups of the maximal tori of $\mathcal A$ and $\mathcal A'$ at the closed point of $S$. Similarly, for each $i=1,\ldots,n$, let $X_i,X'_i$ be the character groups obtained over the generic point of $D_i$.

Let now $\widetilde \phi_1\colon X'_1\to X^{\vee}_1\,\ldots,\widetilde\phi_n\colon X'_n\to X^{\vee}_n$ be the group homomorphisms deduced from the monodromy pairings at the generic points of $D_1,\ldots,D_n$. We can compose them, for each $i=1,\ldots,n$, with the specialization morphism $X'\to X'_i$ and the dual $X^{\vee}_i\to X^{\vee}$. We obtain for every $i=1,\ldots,n$ a group homomorphism $\psi_i\colon X'\to X^{\vee}$. Tensoring with $\Q_l/\Z_l$ be obtain $\widehat\psi_i\colon X'\otimes \Q_l/\Z_l\to X^{\vee}\otimes \Q_l/Z_l$.
These give us another way to describe $_l\Phi$:

\begin{proposition}\label{prop:grp_comps}
Let $l\neq p$; then $_l\Phi \subseteq \bigcap_{i=1}^n\hat\psi_i$, with equality if $n=1$ or $0$. 
\end{proposition}

\begin{proof}
The case $n=0$ is trivial. Recall from \cref{subsubsection_monodromy_pairing} that a choice of topological generator $\sigma_i$ for $I_i$ identifies $\widetilde\phi_i$ with the homomorphism
$$\frac{T_lA(K^s)}{T_lA(K^s)^{I_i}}\to T_lA^{t_i}(K^s)$$
given by applying $\sigma_i-1$. Then $\widehat \phi_i=\widetilde{\phi_i}\otimes\Q_l/\Z_l$ is the homomorphism
$$\frac{T_lA(K^s)\otimes\Q_l/\Z_l}{T_lA(K^s)^{I_i}\otimes\Q_l/\Z_l}\to T_lA^{t_i}(K^s)\otimes\Q_l/\Z_l$$
given by $\sigma-1$. Its kernel is the subgroup $(X'_i\otimes\Q_l/\Z_l)^{I_i}=\frac{(T_lA(K^s)\otimes\Q_l/\Z_l)^{I_i}}{T_lA(K^s)^{I_i}\otimes\Q_l/\Z_l}$ of $I_i$-invariant elements. This settles in particular the case $n=1$, thanks to \eqref{l_primary_phi}.

Now, consider the map $X^{\vee}_i\to X^{\vee}$. It is injective with free cokernel, because $X\to X_i$ is surjective of free abelian groups.  Therefore the induced map $X^{\vee}_i\otimes \Q_l/\Z_l\to X^{\vee}\otimes\Q_l/\Z_l$ is still injective.
It follows that $\ker \widehat\psi_i$ is the preimage of $(X'_i\otimes\Q_l/\Z_l)^{I_i}$ in $X'\otimes \Q_l/\Z_l$, so it certainly contains $(X'\otimes\Q_l/\Z_l)^{I_i}$. In particular, $_l\Phi=(X'\otimes\Q_l/\Z_l)^G= \bigcap_{i=1}^n(X'\otimes\Q_l/\Z_l)^{I_i}\subset \bigcap_{i=1}^n\ker \widehat\psi_i$. 
\end{proof}

\subsubsection{The full group of components}
In the case $\dim S=1$, we have necessarily $n=0$ or $1$, and \cref{prop:grp_comps} extends to describe the $p$-torsion of $\Phi$. Let $\widetilde\phi\colon X'\to X^{\vee}$ be the homomorphism corresponding to the monodromy pairing (\cref{thm:monodromy_pairing}), and consider the restriction of scalars to $\Q/\Z$, $$\widehat{\phi}\colon X'\otimes\Q/\Z\to X^{\vee}\otimes\Q/\Z.$$
\begin{theorem}[\cite{SGA7}, IX.11.4]\label{thm:grp_comps}
The group of components $\Phi$ is naturally identified with $\coker \widetilde{\phi}$ and with $\ker \widehat \phi.$
\end{theorem}
The fact that $\coker \widetilde{\phi}=\ker \widehat \phi$ is an immediate consequence of the fact that $\widetilde\phi$ is an injective homomorphism between free finitely generated abelian groups of the same finite rank, which in turn comes from the fact that the monodromy pairing is non-degenerate.

The proof for the $p$-part relies on the theory of mixed extensions developed in \cite{SGA7}, IX.9.3. As the theory is somewhat complicated, we avoided checking whether proposition \ref{prop:grp_comps} extends to the case $l=p$.
\section{Toric additivity}\label{section3}
In this section we work over a regular, locally noetherian base scheme $S$, with a normal crossing divisor $D\subset S$; we suppose that we are given an abelian scheme $A$ over the open $U=S\setminus D$, of relative dimension $d$, admitting a (unique) semi-abelian model $\mathcal A/S$. We let $A'/U$ be the dual abelian scheme, and $\mathcal A'/S$ the unique semi-abelian scheme extending it.

\subsection{The strictly local case}\label{section:def_toric_add}
Assume that $S$ is strictly local, with fraction field $K$, closed point $s$ and residue field $k=k(s)$ of characteristic $p\geq 0$.
The divisor $D$ has finitely many irreducible components $D_1,\ldots, D_n$. Each $D_i$ has a generic point $\zeta_i$. We let $X$ (resp. $X'$) be the character group of the maximal subtorus of $\mathcal A_s$ (resp. $\mathcal A'_s$), and, for each $1\leq i\leq n$, we let $X_i$ (resp. $X'_i$) be the character group of the maximal subtorus of $\mathcal A_{\zeta_i}$ (resp. $\mathcal A'_{\zeta_i}$). By \cref{lemma:tori_are_split}, all these tori are split and we may consider the character groups as abstract free finitely generated abelian groups.

\subsubsection{Toric additivity} We recall the existence of the injective purity map \ref{def_purity_map}:
$$p\colon X\into X_1\oplus\ldots\oplus X_n.$$

\begin{definition}\label{def_TA}
We say that $\mathcal A/S$ is toric-additive if the purity map
$p$ 
is an isomorphism (or equivalently, is surjective).
\end{definition}

\begin{remark}\label{remark:prop_TA}
Here are some preliminary observations about toric additivity:
\begin{itemize}
    \item[i)] Whether $\mathcal A/S$ is toric-additive depends only on the generic fibre $A_K/K$, since semi-abelian extensions are unique.
    \item[ii)] Suppose that $A/U$ is toric-additive. Let $t$ be another geometric point of $S$, and let $J\subset \{1,\ldots,n\}$ be defined by $j\in J$ if and only if $t\in D_j$. Let $S'\to S$ be the strict henselization at $t$. Let $Y$ be the character group of the maximal torus of $\mathcal A_{t}$. We have a commutative diagram
    \begin{center}
        \begin{tikzcd}
        X \ar[r] \ar[d] & \bigoplus_{i=1,\ldots,n} X_i\ar[d]\\
        Y \ar[r] & \bigoplus_{j\in J}X_j
        \end{tikzcd}
    \end{center}
    where the horizontal maps are the purity maps; the left vertical map is the specialization map (\cref{def:specialization_map}), and the right vertical map is the obvious projection. The two vertical maps are surjective; hence, if the upper horizontal map is surjective, so is the lower one. This shows that if $\mathcal A/S$ is toric additive, so is the base change $\mathcal A_{S'}/S'$. 
    \item[iii)] In the trivial case $n=0$ all character groups vanish, and $\mathcal A/S$ is toric additive.    
    \item[iv)] Suppose $n=1$. Then the injective purity map $X\into X_1$ coincides with the specialization map, which is surjective. Hence in this case $\mathcal A/S$ is automatically toric additive.
\end{itemize}
\end{remark}
\begin{lemma}\label{lemma:A_and_A'}
$\mathcal A/S$ is toric additive if and only if $\mathcal A'/S$ is. 
\end{lemma}
\begin{proof}
Choose a polarization $A\to A'$ of degree $\delta$, which extends uniquely to an isogeny $\lambda \colon \mathcal A\to \mathcal A'$ of degree $\delta$. We let $\lambda'\colon \mathcal A'\to \mathcal A$ be the dual isogeny.

Then $\lambda$ and $\lambda'$ induce isogenies $\lambda^*$ and $\lambda'^*$ on character groups over the closed fibre $s$, of degree $\delta^*\leq \delta$; and also isogenies $\lambda_i\colon X_i\to X'_i$ and $\lambda'_i\colon X'_i\to X_i$ of degree $\delta_i$. The purity maps fit into a commutative diagram
\begin{center}
\begin{tikzcd}
X \ar[d, "p"] \ar[r, "\lambda^*"] & X' \ar[d, "p'"]\ar[r, "\lambda'^*"] & X\ar[d, "p"]\\
\bigoplus X_i\ar[r, "(\lambda_i)"] & \bigoplus X'_i\ar[r, "(\lambda'_i)"] & \bigoplus X_i
\end{tikzcd}
\end{center}
As the purity map $p$ is an isomorphism, the composition of the two lower horizontal maps is multiplication by $(\delta^*)^2$, which has determinant $(\delta^*)^{2\mu}$, where $\mu=\rk X$. Hence both lower horizontal maps have determinant $(\delta^*)^{\mu}$ which is equal to the determinant of $\lambda^*$. It follows that $p'$ is an isomorphism.
\end{proof}

\subsubsection{$l$-toric additivity and the monodromy action}
We keep the hypothesis of \cref{section:def_toric_add}. We fix a prime $l\neq p$ and consider the $l$-adic Tate module $T_lA(K^s)$ where $K^s$ is a separable closure of $K$; we recall that it is a free $\mathbb Z_l$-module of rank $2d$ with an action of $\Gal(K^s|K)$, which factors via the surjection $\Gal(K^s|K)\ra G:=\pi_1^{t,l}(U)=\bigoplus_{i=1}^nI_i$, where $I_i:=\pi_1^{t,l}(S\setminus D_i)=\Z_l(1)$ for each $i$. 

\begin{definition}
Let $l\neq p$ be a prime. We say that $\mathcal A/S$ satisfies condition $\bigstar(l)$ if
\begin{equation}\label{conditionA}
T_lA(K^s)=\sum_{i=1}^nT_lA(K^s)^{\oplus_{j\neq i} I_j} \mbox{ or if } n=0.
\end{equation}
\end{definition}

\begin{remark}\label{remark_TA} $ $ 
\begin{itemize}
\item[i)] Clearly, whether $\mathcal A/S$ satisfies condition $\bigstar(l)$ depends only on the generic fibre $A_K/K$;
\item[ii)] suppose that $\mathcal A/S$ satisfies condition $\bigstar(l)$; let $t$ be another geometric point of $S$ and let $J\subset \{1,\ldots,n\}$ defined by $j\in J$ if and only if $t\in D_j$. Consider the strict henselization $S'$ at $t$. Then the morphism $$\pi_1^{t,l}(U\times_SS')\ra \pi_1^{t,l}(U)$$ induced by $S'\ra S$ is the natural inclusion
$$\bigoplus_{j\in J} I_j\ra \bigoplus_{i=1}^nI_i.$$
It can be easily seen that $\sum_{i=1}^n T_lA(K^s)^{\oplus_{j\neq i} I_j}\subseteq \sum_{j\in J}T_lA(K^s)^{\oplus_{k\in J\setminus\{j\}}I_k}$; hence $\mathcal A_{S'}/S'$ also satisfies condition $\bigstar(l)$.
\item[iii)] Condition $\bigstar(l)$ is automatically satisfied if $n=1$.
\end{itemize}
\end{remark}

\begin{theorem}\label{A=B}
Let $S$ be a regular, strictly local scheme, with closed point $s$ of residue characteristic $p\geq 0$, $D=\bigcup_{i=1}^nD_i$ a normal crossing divisor on $S$. Let $A$ be an abelian scheme over $U=S\setminus D$, of relative dimension $d$, admitting a semiabelian prolongation $\mathcal A/S$. Let $l\neq p$ be a prime and $G=\pi_1^{t,l}(U)=\bigoplus_{i=1}^n I_i$ with $I_i=\Z_l(1)$.

The following conditions are equivalent:
 
 \begin{itemize}
 \item[a)] $\mathcal A/S$ satisfies condition $\bigstar(l)$.
 \item[b)]  The quotient $T_lA(K^s)/T_lA(K^s)^G$ of the Tate module by the subgroup of $G$-invariant elements decomposes canonically into a direct sum of $G$-invariant submodules $V_1,\ldots,V_n$ such that for each $j\neq i$, $I_i\subset G$ acts trivially on $V_j$;
\item[c)] The purity map tensored with $\Z_l$,

$$\phi\otimes_{\Z}\Z_l\colon X\otimes\Z_l\ra (X_1\oplus\ldots\oplus X_n)\otimes\Z_l$$
is an isomorphism.
 \end{itemize}

\end{theorem}

\begin{proof}

To make the notation lighter, we write $T$ in place of $T_lA(K^s)$.
We start by proving that a) implies c). Notice that $\phi\otimes\Z_l$ is an isomorphism if and only if $\phi'\otimes\Z_l$ is an isomorphism, where $\phi'$ is the purity map for $\mathcal A'$; this is proved in the same way as \cref{lemma:A_and_A'}. Now, we need to show that the natural map of $\Z_l$-modules of \cref{map:purity_l}, \cref{subs:purity_map}
$$\alpha \colon T/T^G\ra \bigoplus_{i=1}^n T/T^{I_i}$$ is surjective. Fix an element $\overline t\in T/T^{I_i}$. It is sufficient to show that there is a lift $t^*\in T$ of $\overline t$, such that $t^*$ is fixed by $I_j$ for all $j\neq i$. Pick then an arbitrary lift $t\in T$; by hypothesis we can write it as $t_1+t_2+\ldots+t_n$, with $t_j\in T^{\bigoplus_{k\neq j} I_k}$. Each $t_j$ with $j\neq i$ maps to zero in $T/T^{I_i}$, hence $t_i$ is a lift of $\overline t$. Moreover $t_i$ is fixed by $I_j$ for every $j\neq i$ as we wished to show.

We prove $b)\Rightarrow a)$. Assume that a decomposition of $T$ as in b) exists. Let $W_i\subset T$ be the preimage of $V_i$ via the quotient map $T\to T/T^G$. Then, for every $1\leq i\leq n$, $W_i\subseteq T^{\bigoplus_{j\neq i}I_j}$. Since $T=\sum_{i=1}^nW_i$, condition $\bigstar(l)$ is evidently satisfied; 

We prove $c)\Rightarrow b)$. We consider the natural map of $\Z_l$-modules \cref{map:purity_l}
$$\alpha \colon T/T^G\ra \bigoplus_{i=1}^n T/T^{I_i}$$
which is an isomorphism because $\phi'\otimes\Z_l$ is. We let $V_i=\alpha^{-1}(T/T^{I_i})$. Notice that for every $i=1,\ldots,n$, the inverse morphism $\alpha^{-1}$ identifies $\bigoplus_{j\neq i}T/T^{I_j}$ with the submodule $T^{I_i}/T^G$ of $T/T^G$. Hence $I_j$ acts trivially on $V_i$, for every $j\neq i$. Moreover, $V_i$ is invariant under $I_i$, just by commutativity of the group $G$. In particular $V_i$ is $G$-invariant, which concludes the proof.

\end{proof}

\begin{definition}\label{def_l_TA}
Let $l\neq p$ be a prime. We say that the semi-abelian scheme $\mathcal A/S$ is \textit{$l$-toric-additive} if any of the equivalent conditions of \cref{A=B} is satisfied.
\end{definition}

\begin{remark}
Clearly, if $\mathcal A/S$ is toric additive then it is $l$-toric additivity for all primes $l\neq p$. If the residue characteristic $p$ is zero, then the converse is also true.
\end{remark}

Also for $l$-toric additivity we have the analogous of \cref{lemma:A_and_A'}:
\begin{lemma}\label{lemma:A_and_A'_l}
For a prime $l\neq p$, $\mathcal A/S$ is $l$-toric additive if and only if $\mathcal A'$ is. 
\end{lemma}
\begin{proof}
It is enough to show that the map in part c) of \cref{A=B} is an isomorphism if and only if the analogous map for $\mathcal A$ is an isomorphism. The same proof as for \cref{lemma:A_and_A'} works also in this case.
\end{proof}

\subsubsection{Weak toric additivity}
\begin{definition}\label{def_w_TA}
We say that the semi-abelian scheme $\mathcal A/S$ is \textit{weakly toric-additive} if the inequality of toric ranks in \cref{mu_relation}
$$\mu\leq \mu_1+\mu_2+\ldots+\mu_n$$
is an equality.
\end{definition}

\begin{remark}\label{remark:WTA}
It is clear that $\mathcal A$ is weakly toric additive if and only if $\mathcal A'$ is, since the two share the same toric rank function $\mu$; moreover, if $t$ is a geometric point of $S$ and $S'\to S$ the corresponding strict henselization, $\mathcal A_{S'}/S'$ is still weakly toric additivity. This follows easily from the next \cref{weak_and_l} together with \cref{remark_TA} part ii).
\end{remark}

\subsubsection{Relation between the different notions of toric additivity} While it is evident that toric additivity of $\mathcal A/S$ implies $l$-toric additivity for all primes $l\neq p$ (by comparing ranks in part c) of \cref{A=B}), we see that weak toric additivity is a much weaker property:
\begin{lemma}\label{weak_and_l}
The following are equivalent:
\begin{itemize}
\item[i)] $\mathcal A/S$ is weakly toric additive;
\item[ii)] $\mathcal A/S$ is $l$-toric additive for all but finitely many primes $l\neq p$;
\item[iii)] $\mathcal A/S$ is $l$-toric additive for some prime $l\neq p$.
\end{itemize}
\end{lemma}
\begin{proof}
Weak toric additivity implies that the purity map is between free, finitely generated abelian groups of the same rank, hence its cokernel is finite of some order $N$. Then for all primes $l$ not dividing $pN$ we have that $\mathcal A/S$ is $l$-toric additive, which shows that i) implies ii). The implication ii)$\Rightarrow$ iii) is trivial. Finally, we see that iii) implies i) by taking ranks in \cref{A=B}. 
\end{proof}

\subsubsection{Toric additivity and finite flat base change}

\begin{lemma}\label{finiteflat_TA}
Let $r_1,\ldots,r_n\in \Gamma(S,\O_S)$ be such that $D\subset S$ is the zero locus of $r_1\cdot r_2\cdot\ldots\cdot r_n$.
Let $m_1,\ldots,m_n$ be positive integers and $B$ be the $\Gamma(S,\O_S)$-algebra
 \begin{equation}
B=\frac{\Gamma(S,\O_S)[T_1,\ldots,T_n]}{T_1^{m_1}-r_1,\ldots,T_n^{m_n}-r_n}
 \end{equation}
Write $T=\Spec B$ and let $f\colon T\ra S$ be the induced morphism of schemes. let $l\neq p$ be a prime. Then $\mathcal A/S$ is toric-additive, resp. $l$-toric additive, resp. weakly toric additive, if and only if $\mathcal A_T/T$ is toric additive, resp. $l$-toric additive, resp. weakly toric additive.
 \end{lemma}
 \begin{proof}
Notice that $T$ is a regular strictly local scheme, so all notions of toric additivity for $\mathcal A_T/T$ make sense. The statement follows easily from the fact that $f^{-1}(D)_{red}\ra D$ is an isomorphism.
 \end{proof}

\subsection{Global definition of toric additivity}
We have defined toric-additivity over a strictly local base. We now remove this hypotheses and consider hypotheses as in the beginning of \cref{section3}.

\begin{definition}\label{TAglobal}
We say that $\mathcal A/S$ is \textit{toric additive} at a geometric point $s$ of $S$, if the base change $\mathcal A\otimes_S\Spec\O^{sh}_{S,s}$ to the strict henselization at $s$ is toric additive as in \cref{def_TA}. We say that $\mathcal A/S$ is \textit{toric additive} if it is so at all geometric points $s$ of $S$.

We define $\mathcal A/S$ to be $l$-\textit{toric additive} for some $l$ invertible on $S$, resp. \textit{weak toric additive}, in an analogous way.
\end{definition}

\begin{remark}\label{remark:TA_etale_local}
It follows from the definition that toric additivity, $l$-toric additivity, weak toric additivity are property \'etale-local on the target. 
\end{remark}

We actually have the stronger statement:
\begin{lemma} \label{TAsmooth}
Toric-additivity, $l$-toric additivity and weak toric additivity are local on the target for the smooth topology.
\end{lemma}

 \begin{proof}
First, notice that for any $n\geq 1$, $\mathcal A/S$ is toric additive if and only if its base change via $\mathbb A^n_S\to S$ is. Now, given $f\colon T\ra S$ smooth and surjective, and $\bar t$ a geometric point lying over $t\in T$ and over $f(t)\in S$, there is a commutative diagram
\begin{center}
\begin{tikzcd}
T \ar[d] & W \ar[l]\ar[r] & \mathbb A^N_{V} \ar[d]\\
S & &\ar[ll] V
\end{tikzcd}
\end{center} 
where $W$ is an open neighbourhood of $t$ and $V$ an open neighbourhood of $f(t)$, and the map $W\mapsto \mathbb A^N_V$ is \'etale. Now the statement follows from \cref{remark:TA_etale_local}.

 \end{proof}
 
\begin{lemma}\label{TA_open_condition}
Toric-additivity, $l$-toric additivity and weak toric additivity of $\mathcal A/S$ are open conditions on $S$.
\end{lemma}
\begin{proof}
Suppose that $\mathcal A/S$ is toric-additive at a geometric point $s$. It is enough to show that $\mathcal A/S$ is toric-additive on an \'etale neighbourhood of $s$, since \'etale morphisms are open. We may therefore assume that $D$ is a strict normal crossing divisor and that $s$ belongs to all irreducible components $D_1,\ldots, D_n$ of $D$. 

Let $t$ be another geometric point of $S$; we want to show that $\mathcal A/S$ is toric-additive at $t$. This is true if $t\not\in D$, so we may assume without loss of generality that $t$ belongs to $D_1,\ldots,D_m$ for some $1\leq m\leq n$. Let $\zeta$ be a geometric point lying over the generic point of $D_1\cap D_2\cap \ldots\cap D_m$; write $S_{\zeta},S_t,S_s$ for the spectra of the strict henselizations of $S$ at $\zeta,t,s$ respectively. The morphism $S_{\zeta}\ra S$ factors via $S_{\zeta}\ra S_s$, the strict henselization of $S_s$ at $\zeta$; hence, by \cref{remark:prop_TA}, $\mathcal A_{S_{\zeta}}/S_{\zeta}$ is toric-additive. 

Consider now the natural map $S_{\zeta}\ra S_t$.  Let $Y$ be the character group at $t$, $X$ the character group at $\zeta$, and $Y_1,\ldots,Y_m$ the character groups at the generic points of $D_1,\ldots,D_m$. We have natural maps
$$Y\to X\to Y_1\oplus\ldots\oplus Y_m$$
where the first arrow is the surjective specialization map, the second arrow is the purity map for $S_{\zeta}$ and the composition is the purity map for $S_t$. By toric-additivity of $\mathcal A/S_{\zeta}$, the second arrow is an isomorphism. As the composition of the two maps is injective, we see that $Y\to X$ is injective, hence an isomorphism. Hence the purity map for $S_t$ is an isomorphism, as we wished to show.

The same proof works for $l$-toric additivity when taking character groups tensored with $\mathbb Z_l$. For weak toric additivity, we can restrict to an open neighbourhood of $s$ such that the residue characteristic of $s$ is the only non-invertible prime, and then apply \ref{weak_and_l}.
\end{proof}

\begin{lemma}
Let $B/U$ be an abelian subgroup scheme of $A/U$ with a semiabelian prolongation $\mathcal A\into \mathcal B$ over $S$. If $\mathcal A/S$ is toric additive, resp. $l$-toric additive for some invertible prime $l$, resp. weakly toric additive, then so is $\mathcal B/S$.
\end{lemma}
\begin{proof}
We can reduce to $S$ strictly henselian, with closed point $s$. Let $X,Y$ be the character groups of the maximal tori of $A_s$ and $B_s$; and similarly let $X_i,Y_i$ be the character groups at the generic points of $D$. We have a commutative diagram of purity maps
\begin{center}
\begin{tikzcd}
X \ar[r]\ar[d] & Y\ar[d]\\
\bigoplus X_i\ar[r] & \bigoplus Y_i
\end{tikzcd}
\end{center}
where the vertical maps are the purity maps. As the purity map $X\to \bigoplus X_i$ is surjective, and as well all maps $X_i\to Y_i$ induced by the injection $\mathcal A\to \mathcal B$, we conclude that $Y\to\bigoplus Y_i$ is surjective. This shows the part of the statement concerning toric additivity. The other two parts follow by considering character groups tensored with $\mathbb Z_l$.
\end{proof}

\subsubsection{Three examples}
Let $k$ be an algebraically closed field, $S=\Spec k[[u_1,u_2]]$, and let $D$ be the vanishing locus of $u_1u_2$. 
\begin{example}
Consider the nodal projective curve $\mathcal E\subset \mathbb P^2_{S}$ given by the equation 
$$Y^2Z=X^3-X^2Z-u_1u_2Z^3.$$
The restriction $\mathcal E_U/U$, with its section at infinity, is an elliptic curve, and therefore naturally identified with its jacobian $\Pic^0_{\mathcal E_U/U}$; the smooth locus $\mathcal E^{sm}/S$ has a unique $S$-group scheme structure extending the one of $\mathcal E_U/U$, and is a semi-abelian scheme.

Let $\zeta_1,\zeta_2$ be the generic points of $D_1=\{u_1=0\}$ and $D_2=\{u_2=0\}$ respectively, and let $s$ be the closed point $\{u_1=0,u_2=0\}$. The fibres of $\mathcal E^{sm}$ over $\zeta_1$, $\zeta_2$, $s$ are all tori of dimension $1$. It follows that $\mathcal E^{sm}$ is not weakly toric additive.

\end{example}
 \begin{example}
Consider for $i=1,2$ the nodal projective curve $\mathcal E_i\subset \mathbb P^2_{S}$ given by the equation 
$$Y^2Z=X^3-X^2Z-u_iZ^3.$$
In this case the restriction of $\mathcal E_i$ to $S\setminus D_i$, with its infinity section, is an elliptic curve. The smooth locus $\mathcal E_i^{sm}/S$ has a unique structure of semiabelian scheme extending the group structure on $\mathcal E_{i,U}$. Let $\mu_i$ the toric-rank function for $\mathcal E_i/S$. Then $\mu_i(s)=1$ for $i=1,2$, while
$$\mu_i(\zeta_j)=
\begin{cases}
1 & \mbox{if } i=j \\
0 & \mbox{if } i\neq j
\end{cases} 
$$
Hence each $\mathcal E_i/S$ is toric additive, and so is the $S$-semiabelian scheme of relative dimension two $\mathcal E_1\times_S\mathcal E_2$.
\end{example}

\begin{example}
Consider the nodal curve $C$ given by taking $\mathbb P^1_{\C}$, four closed points $P_1,P_2,Q_1,Q_2$ and identifying $P_1$ and $P_2$ in a node $P$ and $Q_1$ and $Q_2$ in a node $Q$. Let $\mathcal C/S$ be its universal deformation as a stable curve. Then $S$ is local of dimension $3g(C)-3=3$, with closed point $s\in S$ and there is a normal crossing divisor $D\subset S$ such that $\mathcal C/S$ is smooth over $S\setminus D$. There are two irreducible components $D_1, D_2\subset D$ each one corresponding to the locus on $S$ where one of the two singular points of $C$ is preserved. The fibre of $\mathcal C/S$ over any of the two generic point $\zeta_1,\zeta_2$ of $D$ is given by pinching a smooth genus $1$ curve to itself. Hence, $\mu(\zeta_1)=\mu(\zeta_2)=1$, while $\mu(s)=2$, so $\mathcal C/S$ is weakly toric additive. In view of \cref{TA_and_weak_curves}, it is actually toric additive. 
\end{example}
 
\subsection{Comparison with toric additivity of jacobians}\label{subs:case_of_curves}
The notion of toric additivity was originally introduced in \citep{Orecchia} in the case where $\mathcal A/S$ is $\Pic^0_{\mathcal C/S}$ for a nodal curve $\mathcal C/S$ smooth over $U=S\setminus D$. In the case of jacobians, no distinction between toric additivity and weak toric additivity is made, as they are equivalent:

\begin{lemma}[\citep{Orecchia}, 2.5]\label{TA_and_weak_curves}
Let $\mathcal C/S$ be a nodal curve, smooth over $U$, and let $\mathcal A=\Pic^0_{\mathcal C/S}$. Then $\mathcal A/S$ is toric additive if and only if it is weakly toric additive.
\end{lemma}
The lemma follows immediately from the fact that for jacobians of curves, the cokernel of the purity map is always torsion-free (\citep{Orecchia} 2.4).

We deduce:
\begin{corollary}\label{coro:curves_l}
Let $\mathcal C/S$ be a nodal curve, smooth over $U$, and let $\mathcal A=\Pic^0_{\mathcal C/S}$. The following are equivalent:
\begin{itemize}
\item[i)] $\mathcal A/S$ is weakly toric additive; 
\item[ii)] for some prime $l\neq p$, $\mathcal A/S$ is $l$-toric additive; 
\item[iii)] for every prime $l\neq p$, $\mathcal A/S$ is $l$-toric-additive.
\item[iv)] $\mathcal A/S$ is toric additive. 
\end{itemize}
\end{corollary}

The motivation for the introduction of toric additivity is that it serves as a criterion for existence of N\'eron models, by the following theorem:
\begin{theorem}[\citep{Orecchia}, 4.13]
Let $S$ be a regular, excellent scheme, $D\subset S$ a normal crossing divisor, $\mathcal C/S$ a nodal curve smooth over $U=S\setminus D$. Then the equivalent conditions of \cref{coro:curves_l} are equivalent to the existence of a N\'eron model for $\Pic^0_{\mathcal C_U/U}$ over $S$.
\end{theorem}

We deduce the following corollary:
\begin{corollary}
Let $\mathcal C, \mathcal D$ be nodal curves over $S$, smooth over $U$, such that the jacobians $\mathcal P=\Pic^0_{\mathcal C_U/U}$ and $\mathcal Q=\Pic^0_{\mathcal D_U/U}$ are isogenous. Then $\mathcal P$ admits a N\'eron model over $S$ if and only if $\mathcal Q$ does.
\end{corollary}
\begin{proof}
It is enough to check the statement when $S$ is local, strictly henselian. Let $f\colon P\to Q$ be an isogeny, and pick a prime $l$ different from the residue characteristic of $S$ and not dividing the degree of $f$. Then the induced map $f\colon T_lP(K^s)\to T_lQ(K^s)$ is an isomorphism of Galois-modules. In particular $P$ is $l$-toric additive if and only if $Q$ is.
\end{proof}

\subsection{An example}\label{example:uniformization}
It is natural to ask whether \cref{TA_and_weak_curves} holds for abelian schemes as well. Surprisingly, the answer is negative. We present a counterexample where the purity map has torsion cokernel:
\begin{example}
Let $k=\mathbb C$ and $S$ the spectrum of the local ring of dimension two $R=k[[u,v]]$, which is complete with respect to its maximal ideal $\mathfrak m=(u,v)$. We denote by $K$ its fraction field $k((u,v))$. We consider the split torus over $S$ of rank two $\widetilde G=\mathbb G_{m,S}^2=\Spec R[x,x^{-1}]\times \Spec R[y,y^{-1}]$, with character group $X=\mathbb Z^2_S$ generated by $x$ and $y$. We choose a set of periods 
$$Y=\langle a,b \rangle \subset \widetilde G(K)=K^{\times}\times K^{\times}\;\; a=(u^4,u^2), b=(u^2,uv).$$
The groups of periods $Y$ admits a principal polarization
\begin{align*}
\lambda\colon Y &\to X \\
a &\mapsto x\\ b &\mapsto y
\end{align*}
Using Mumford's theory of relatively complete models, the data of $\widetilde G$ and $Y$ yield a principally polarized semiabelian scheme $G/S$, with restriction $G_U/U$ to $U=S\setminus\{uv=0\}$ an abelian scheme and with $G_k=\widetilde G_k$. 

We are going to show that the $G/S$ is weakly toric additive but not toric additive. The polarization $\lambda$ identifies $X$ with the set of periods $Y$, which is in turn naturally identified with the character group of the maximal torus of the fibre of $G^\vee/S$ over the closed point.

Consider, for every $(a,b)\in \Z_{\geq 1}^2$, the trait $Z_{a,b}=\Spec k[[t]]$, and the morphism $f_{a,b}\colon Z\to S$, given by $u\to t^a, v\to t^b$. We denote by $\phi_{a,b}\colon X \to X$ the monodromy pairing of the semiabelian scheme $f_{a,b}^{-1}G$ over $Z$. The latter is determined by the set of periods $Y_{a,b}\subset \tilde G(k((t)))$, $Y_{a,b}=\langle (t^{4a},t^{2a}),(t^{2a},t^{a+b})\rangle$. It follows that \begin{equation}
\phi_{a,b}=\begin{pmatrix}
4a & 2a \\
2a & a+b
\end{pmatrix}
\end{equation}

On the other hand, $\phi_{a,b}$ can also be written as

\begin{center}
\begin{tikzcd}
X\ar[r, "p"] & X_1\oplus X_2\ar[r,"\psi_{a,b}"] & X_1\oplus X_2 \ar[r,"p^t"] & X
\end{tikzcd}
\end{center}

where: $X_1\cong \Z,X_2\cong \Z$ are the character groups of the maximal subtori of the fibres of $G/S$ over the generic points $\zeta_1,\zeta_2$ of $\{u=0\}$ and $\{v=0\}$; $p$ and $p^t$ are the purity map and its transpose; the central arrow is
$$
\psi_{a,b}=\begin{pmatrix}
a\phi_1 & 0 \\ 0 & b\phi_2
\end{pmatrix}
$$
with $\phi_1,\phi_2$ the monodromy pairings of the restriction of $G/S$ to the local rings at $\zeta_1$ and $\zeta_2$.  

Letting 
$$ 
p=\begin{pmatrix}
A & B\\ 
C & D
\end{pmatrix}
$$
we find

$$
\psi_{a,b}=\begin{pmatrix}
aA^2\phi_1+bC^2\phi_2 & aAB\phi_1+bCD\phi_2\\
aAB\phi_1+bCD\phi_2 & aB^2\phi_1+bD^2\phi_2
\end{pmatrix}=
\begin{pmatrix}
4a & 2a \\
2a & a+b
\end{pmatrix}
$$
A calculation shows that $A=2,B=1,C=0,D=1$, that is, 
$$p=\begin{pmatrix}2 & 1 \\ 0 & 1\end{pmatrix}.$$

We see immediately that the $p$ is not surjective, and has finite cokernel of order $2$. This shows that $G/S$ is $l$-toric additive for all primes $l\neq 2$, and in particular is weakly toric additive, but it is not $2$-toric additive. 
\end{example}

\newpage
\section{N\'eron models of abelian schemes in characteristic zero}\label{section4}

In this section, we consider a locally noetherian, regular base scheme $S$, a normal crossing divisor $D$ on $S$, an abelian scheme $A/U$ of relative dimension $d$ and a semi-abelian scheme $\mathcal A/S$ with a given isomorphism $\mathcal A\times_SU\ra A$. We will retain notations used in the previous sections.

\subsection{Test-N\'eron models}
\begin{definition}\label{def:test_NM}
Let $\mathcal N/S$ be a smooth, separated group algebraic space of finite type with an isomorphism $\mathcal N\times_SU\ra A$; we say that it is a \textit{test-N\'eron model} for $A$ over $S$ if, for every strictly henselian trait $Z$ and morphism $Z\ra S$ transversal to $D$ (\cref{trait_transversal}), the pullback $\mathcal N\times_SZ$ is the N\'eron model of its generic fibre. 
\end{definition}

It is clear that the property of being a test-N\'eron model is smooth-local on the base, and is also preserved by taking the localization at a point of the base, or the strict henselization at a geometric point. 

\begin{lemma}\label{lemma:test_fibconn}
The fibrewise-connected component of identity $\mathcal N^0$ of $\mathcal N$ is semi-abelian and is naturally identified with $\mathcal A/S$.
\end{lemma}
\begin{proof}
Take a point $s\in S$ and a morphism $Z\to S$ from a strictly henselian trait, transversal to $D$, mapping the closed point of $Z$ to $s$. Then $\mathcal N_Z/Z$ is a N\'eron model of its generic fibre and in particular $\mathcal N^0_Z/Z$ is semiabelian. Hence $\mathcal N^0_s$ is semiabelian. By uniqueness of semi-abelian extensions, $\mathcal N^0$ is naturally idenfitied with $\mathcal A$. 
\end{proof}

\subsubsection{Uniqueness} Just as N\'eron models are unique up to unique isomorphism, the same holds for test-N\'eron models:
\begin{proposition}\label{test-NM_unique}
If $\mathcal M/S$ and $ \mathcal N/S$ are two test-N\'eron models for $A$, there exists a unique isomorphism $\mathcal M\ra \mathcal N$ that restricts to the natural isomorphism $\mathcal M^0\ra \mathcal N^0$ of \cref{lemma:test_fibconn}.
\end{proposition}
\begin{proof}
The uniqueness is automatic, because $\mathcal N$ is separated and $\mathcal M^0$ is schematically dense in $\mathcal M$. For the existence part, we proceed by induction on the dimension of the base. In the case of $\dim S=1$, let $S^{sh}$ be a strict henselization of the trait $S$. The base change of a test-N\'eron model to $S^{sh}$ is a N\'eron model. By \cref{fpqclocal}, $\mathcal M$ and $\mathcal N$ are themselves N\'eron models over $S$, and therefore there exists an isomorphism $\mathcal M\ra \mathcal N$. 

Now let $\dim S=M$ and assume the statement is true for $\dim S<M$. We claim that we can reduce to the case of a strictly henselian local base $S$. Suppose that for every geometric point $s$ of $S$ we can construct an isomorphism $f_s\colon\mathcal M_{X_s} \ra \mathcal N_{X_s}$ where $X_s$ is the spectrum of the strict henselization at $s$. Then we can spread out $f_s$ to an isomorphism $f'\colon \mathcal M_{S'}\ra \mathcal N_{S'}$ for some \'etale cover $S'$ of $S$. Let $S'':=S'\times_SS'$, $p_1,p_2\colon S''\ra S'$ be the two projections and $q\colon S''\ra S$. Because test-N\'eron models are stable under \'etale base change, $q^*\mathcal M$ and $q^*\mathcal N$ are test-N\'eron models. The two isomorphisms $p_1^*f,p_2^*f\colon q^*\mathcal M\ra q^*\mathcal N$ necessarily coincide, thus $f$ descends to an isomorphism $\mathcal M\ra \mathcal N$, which proves our claim.

Let then $S$ be strictly local, of dimension $M$, with closed point $s$ and fraction field $K$. The open $V=S\setminus\{s\}$ has dimension $M-1$; by inductive hypothesis there is a unique isomorphism $f_V\colon \mathcal M_V\ra \mathcal N_V$. We would like to extend it to the whole of $S$. 

Let $Z$ be a regular, closed subscheme of $S$ of dimension $1$, transversal to $D$. The existence of such $Z\subset S$ is easily checked. We call $L$ the fraction field of $\Gamma (Z,\mathcal O_Z)$. As $Z$ is a strictly henselian trait, the pullbacks of $\mathcal M$ and $\mathcal N$ to $Z$ are N\'eron models of their generic fibres $\mathcal M_L$ and $\mathcal N_L$; the isomorphism $\mathcal M_L\ra \mathcal N_L$ induces a unique isomorphism $\alpha\colon\mathcal M_Z\ra \mathcal N_Z$. Now let $\underline{\Phi}$ and $\underline{\Psi}$ be the \'etale $S$-group schemes of components of $\mathcal M$ and $\mathcal N$; and let $\Phi$ and $\Psi$ be the finite abelian groups $\underline \Phi_s(k)$ and $\underline \Psi_s(k)$ respectively, where $k$ is the (separably closed) residue field of $S$. The restriction of $\alpha$ to the fibre over $s$ induces an isomorphism $\alpha_Z\colon \Phi\ra \Psi$, a priori depending on the choice of the trait $Z$.

We show that the isomorphism $\Phi\ra \Psi$ is independent of the choice of $Z\subset S$. We may reduce to the case where $\Phi$ and $\Psi$ are isomorphic to $\Z/q^m\Z$, for some prime $q$ and $m\geq 0$.

Let $\phi_Z\colon X'\ra (X)^{\vee}$ be the homomorphism induced by the monodromy pairing (cf. \cref{thm:monodromy_pairing}) for $\mathcal M^0_{Z}$, and $\psi_Z\colon Y'\ra Y^{\vee}$ the one for $\mathcal N^0_Z$. Then $\Phi$ (resp. $\Psi$) is naturally identified with the kernel of $\phi_Z\otimes \Z/q^m\Z$ (resp. $\psi_Z\otimes\Z/q^m\Z$) by \cref{thm:grp_comps}. The natural isomorphism $\mathcal M^0\ra \mathcal N^0$ induces an isomorphism between the duals $\mathcal N'^0\to \mathcal M'^0$ and therefore isomorphisms $X'\ra Y'$ and $Y\ra X$, compatible with the monodromy pairings; we get a commutative diagram
\begin{center}
\begin{tikzcd}
\Phi \arrow[r]\arrow[d, "\alpha_Z"] & X'\otimes\Z/q^m\Z \arrow[rr, "\phi_Z\otimes \Z/q^m\Z"]\arrow[d]&& X^{\vee}\otimes\Z/q^m\Z \arrow[d]\\
\Psi \arrow[r] & Y'\otimes\Z/q^m\Z \arrow[rr, "\psi_Z\otimes\Z/q^m\Z"] && Y^{\vee}\otimes\Z/q^m\Z
\end{tikzcd}
\end{center}

We only need to show that the maps $\phi_Z$ and $\psi_Z$ do not depend on the choice of transversal trait $Z$. This is indeed true by \cref{prop:monodromy_transversal}. Hence the isomorphism $\Phi\ra \Psi$ is independent of the choice of $Z\subset S$, hence canonical. For this reason, we will write $\Phi$ for both groups $\Phi$ and $\Psi$.

Let now $q$ be a prime and $m$ the biggest integer for which $q^m$ divides the order of $\Phi$. We write $\underline\Phi_K$ (resp. $X'_K$) for the constant group scheme over $K$ with value $\Phi$ (resp. $X'$) and $\underline\Phi_K[q^m]$ for its $q$-torsion part. By \cref{subs:char_grps} and \cref{thm:grp_comps} we have an injective map 
$$\underline\Phi_K[q^m]\hookrightarrow X'_K\otimes\Z/q^m\Z=\frac{A_K[q^m]}{T_qA^f\otimes\Z/q^m\Z}$$ 
As $X'_K\cong \underline{\Z}^{\mu}_K$ for some $\mu\geq 0$, there exists a section of the surjective map of $q$-divisible groups
$$T_qA_K\ra \frac{T_qA_K}{T_qA_K^f}=X'_K\otimes \Z_q.$$
We pick such a section and obtain a closed immersion of finite group schemes 
$\underline\Phi_K[q^m]\ra A_K[q^m]$; notice that $\underline \Phi_K[q^m]$ is \'etale over $K$ (in fact, a disjoint union of copies of $\Spec K$), while $A_K[q^m]$ is not if $p=q$. Repeating the argument for all primes dividing $N=\ord (\Phi)$, we obtain a closed immersion of finite group schemes $\underline \Phi_K\to A_K[N]$. We denote by $B\subset A_K[N]$ the closed subscheme defined by the closed immersion.

Consider the schematic closure $\mathcal B$ of $B$ inside $\mathcal M$; we claim that $\mathcal B\ra S$ is a disjoint union of open immersions. First, $\mathcal B$ is contained in $\mathcal M[N]$ which is quasi-finite and separated, hence it is itself quasi-finite and separated. By Zariski's main theorem, there exists a factorization $\mathcal B\ra T\ra S$, with the first map an open immersion and the second a finite morphism. Because $T_K$ is finite over $K$, it is discrete, hence $B=\mathcal B_K$ is open and closed in $T_K$. The schematic closure $Y$ of $B_K$ inside $T$ is still finite over $S$; moreover $\mathcal B\ra Y$ is an open immersion and $\mathcal B_K=Y_K$. Therefore we may replace $T$ by $Y$ and simply assume that the factorization $\mathcal B\ra T\ra S$ is such that $\mathcal B_K=T_K$. Now, $$\mathcal B_K=\sqcup_{\phi\in \Phi}\Spec K.$$ Because $S$ is regular, every finite birational morphism to $S$ is an isomorphism; hence $$T=\sqcup_{\phi\in \Phi}S.$$ We deduce that $\mathcal B=\sqcup_{\phi\in \Phi}V_{\phi}$ is indeed a disjoint union of open subschemes of $S$.

Next, we claim that $\mathcal B\ra S$ is finite, i.e. $V_{\phi}=S$ for all $\phi\in \Phi$. As $\mathcal M[N]$ is finite over $U$, the restriction of the open immersion $V_{\phi}\ra S$ to $U$ is itself finite, hence an isomorphism. In particular, it is given by some section $U\ra A$, which restricts to a section $\Spec L\ra A_{\Spec L}$ over the generic point of the trait $Z$. As $\mathcal M_Z$ is a N\'eron model of its generic fibre, this section extends to a section $Z\ra \mathcal M_Z$. This latter section is contained in the schematic closure of $V_{\phi}$, which is $V_{\phi}$ itself. This shows that $V_{\phi}\ra S$ is surjective, and in particular an isomorphism, as claimed.

Therefore, $\mathcal B$ is simply given by a disjoint union $\sqcup_{\phi\in\Phi}b_{\phi}$ of torsion sections $b_{\phi}\colon S\ra \mathcal M$, and the restriction $\mathcal B_s$ is naturally isomorphic to $\underline \Phi_s$.

We construct $\mathcal B'$ as the schematic closure of $B$ inside $\mathcal N$; similarly, we write $\mathcal B'=\sqcup_{\phi\in\Phi}b'_{\phi}$.

Now, recall that $\mathcal A$ is naturally identified with the fibrewise connected components $\mathcal M^0$ and $\mathcal N^0$. Let $\mathcal H=\bigcup_{\phi\in\Phi} (b_{\phi}+\mathcal A)\subseteq \mathcal M$. It is an open subgroup $S$-scheme of $\mathcal M$, and on the closed fibre we have $\mathcal H_s=\mathcal M_s$, since $\mathcal B_s=\underline \Phi_s$. In particular, $\mathcal M=\mathcal M_V\cup \mathcal H$ (recall that $V=S\setminus\{s\}$). Writing similarly $\mathcal H'=\bigcup_{\phi\in\Phi}(b'_{\phi}+\mathcal A)\subseteq \mathcal N$, we have $\mathcal N=\mathcal N_V\cup \mathcal H'$. 

Now, we construct an isomorphism $\mathcal H\ra \mathcal H'$ simply by sending $b_{\phi}$ to $b'_{\phi}$. To obtain an isomorphism $\mathcal M\ra \mathcal N$ it is enough to show that the given isomorphism $f_V\colon \mathcal M_V\ra \mathcal N_V$ and $\mathcal H\ra \mathcal H'$ agree on the intersection $\mathcal M_V\cap \mathcal H=\mathcal H_V$. This is clear: indeed, the isomorphism $f_V\colon\mathcal N_V\ra \mathcal N'_V$ is the unique isomorphism extending the identity morphisms between fibrewise connected components $\mathcal A_V\ra \mathcal A_V$, and it sends the schematic closure of $B$ inside $\mathcal M_V$ to the schematic closure of $B$ inside $\mathcal N_V$; that is, it restricts to an isomorphism $\mathcal B_V\ra \mathcal B'_V$ sending $b_{\phi}$ to $b'_{\phi}$. This concludes the proof.
\end{proof}

\subsubsection{Existence} We pass now to show that under the assumption of toric additivity, test-N\'eron models exist.
\begin{theorem} \label{main_thm}
Suppose that $\mathcal A/S$ is toric-additive. Then there exists a test-N\'eron model $\mathcal N/S$ for $A$.
\end{theorem}

\begin{proof} 
Our proof is constructive; we subdivide it in steps.

\textbf{Step 1: constructing the group $\Psi$}. Let $s$ be a geometric point of $S$, and write $V_s$ for the spectrum of the strict henselization at $s$. Let $K_s$ be the field of fractions of $V_s$, that is, the maximal extension of $K$ unramified at $s$. Let $\mathcal J_s=\{D_1,\ldots,D_n\}$ be the finite set of components of the strict normal crossing divisor $D\times_SV_s$. For every $i\in \mathcal J_s$, let $\zeta_i$ be the generic point of $D_i$, and $X_i$ the character group of the maximal torus of $\mathcal A_{\zeta_i}$. Let also $X$ be the character group of the maximal torus of $\mathcal A_s$. For every $i\in \mathcal J_s$, the monodromy pairing at $\zeta_i$ induces a group homomorphism
$$X'_i\otimes\Q/\Z\ra X_i^{\vee}\otimes\Q/\Z$$
with kernel $\Phi_i$, the (finite) group of components of the N\'eron model of $\mathcal A_{\mathcal O_{\zeta_i}}$. We define the finite abelian group
$$\Psi:=\bigoplus_{i\in\mathcal J_s}\Phi_i\subset \left(\bigoplus_{i\in\mathcal J_s} X'_i\right)\otimes\Q/\Z=X'\otimes\Q/\Z$$
where the last equality is deduced from the purity map, which is an isomorphism under our assumption of toric-additivity. We will write $N$ for the order of $\Psi$, so that $\Psi\subset X'\otimes\Z/N\Z$.

We write $\underline X'$ and $\underline \Psi$ for the constant group schemes over $V_s$ with value $X'$ and $\Psi$ respectively; like in the proof of \cref{test-NM_unique} we may choose a section of the surjection $A_{K_s}[N]\ra \underline X'_{K_s}\otimes \Z/N\Z$, and therefore obtain an injective map of group schemes $\alpha\colon \Psi_{K_s}\ra A_{K_s}[N]$. We may see this as a group homomorphism $\Psi=\Psi_{K_s}(K_s)\ra A_{K_s}[N](K_s)$; we remark that it depends on the choice of section $X'_{K_s}\otimes \Z/N\Z\ra A_{K_s}[N]$.

Consider now the set of sections $\mathcal S=\{\underline X'_{K_s}(K_s)\otimes\Z/N\Z\ra A_{K_s}[N](K_s)\}$. It is a torsor under the finite group $\bigoplus_{l|N}T_lA^f(K_s)\otimes\Z/N\Z$, and as such it is finite. As the group $\Psi$ is finite as well, there exists a finite extension $K\ra K'$, unramified over $s$, such that for any choice of section in $\mathcal S$, the associated map  $\Psi\ra A[N](K_s)$ factors via $A[N](K')$.

\textbf{Step 2: spreading out to an \'etale neighbourhood of $s$.} The normalization of $S$ inside $K'$ is unramified over the image of $s$ in $S$, hence \'etale over it (\citep{stacks}\href{http://stacks.math.columbia.edu/tag/0BQK}{TAG 0BQK}), so we obtain an \'etale neighbourhood $S'$ of $s$, which we may assume to be connected, with fraction field $K'$. We write $\mathcal J'$ for the set of irreducible components of $D\times_SS'$. There is a natural function $\mathcal J_s\ra \mathcal J'$: up to restricting $S'$, we may assume that it is bijective. Indeed, its surjectivity corresponds to the fact that every component of $D\times_SS'$ contains (the image of) $s$; imposing also injectivity means asking that $D\times_SS'$ is a \textit{strict} normal crossing divisor.  Thus, we need not distinguish between $\mathcal J_s$ and $\mathcal J'$ and we will simply write $\mathcal J$ for this set.

\textbf{Step 3: constructing the subgroup-scheme $\mathcal H\subseteq \mathcal A_{S'}\times_{S'}\Psi_{S'}.$} We call $H\subseteq A[N](K')\times \Psi$ the image of $\Psi$ via $(\alpha,\id)\colon \Psi\rightarrow A[N](K')\times \Psi$; we let $\mathcal H/S'$ be the schematic closure of $H$ inside $\mathcal A_{S'}\times_{S'} \underline \Psi$ (where by $\underline \Psi$ we denote the constant group scheme over $S'$ associated to the finite abelian group $\Psi$). We may argue as in the proof of \cref{test-NM_unique} to show that $\mathcal H$ is \'etale over $S$ and actually it is a disjoint union $\sqcup_{j\in \Psi} V_j\rightarrow S'$ of open immersions. In fact, if we write $U'=U\times_SS'$, the base change $A_{U'}$ is an abelian scheme; therefore $A_{U'}[N]\times_{U'}\underline \Psi_{U'}$ is finite, and each $V_j\ra S'$ is an isomorphism over $U'$. This can be restated by saying that  the composition $$\mathcal H_{U'}\ra \mathcal A_{U'}\times_{U'}\Psi_{U'}\ra \underline\Psi_{U'}$$ is an isomorphism.  

\textbf{Step 4: taking the quotient by $\mathcal H$.} Consider now the fppf-quotient $$\mathcal N^{\alpha}:=\frac{\mathcal A_{S'}\times_{S'} \Psi_{S'}}{\mathcal H}.$$
The superscript $\alpha$ reminds us that $\mathcal H$ depends on the choice of section $\alpha$. First, we claim that its restriction $\mathcal N^{\alpha}_{U'}$ is canonically isomorphic to $A_{U'}$. Indeed, we observed that $\mathcal H_{U'}=\underline\Psi_{U'}$, and the quotient morphism for $\underline\Psi_{U'}\ra \mathcal A_{U'}\times_{U'}\underline\Psi_{U'}$, $\psi\mapsto (\alpha(\psi),\psi)$ is $\mathcal A_{U'}\times_{U'}\underline\Psi_{U'}\ra \mathcal A_{U'}$, $(a,\psi)\mapsto a-\alpha(\psi)$, which proves the claim.

Because $\mathcal H$ is \'etale, $\mathcal N^{\alpha}$ is automatically an algebraic space; we claim that it is actually representable by a scheme. As the quotient morphism $p\colon \mathcal A_{S'}\times_{S'}\underline \Psi\ra \mathcal N^{\alpha}$ is an $\mathcal H$-torsor, $p$ is \'etale. In particular the restriction of $p$ to the connected component of identity, $\mathcal A_{S'}\times\{0\}\ra \mathcal N^{\alpha}$, is \'etale; it is also separated, and an isomorphism over $U'$. It follows that it is an open immersion. Hence, all other components $\mathcal A_{S'}\times \{\psi\}$ map to $\mathcal N^{\alpha}$ via an open immersion. The disjoint union $\bigsqcup_{\psi\in\Psi}\mathcal A_{S'}\times_{S'}\{\psi\}$ surjects onto $\mathcal N^{\alpha}$, and this gives us an open cover of $\mathcal N^{\alpha}$ by schemes. 

In summary, we have obtained an $S'$-group scheme $\mathcal N^{\alpha}$, which restricts to $A$ over $U'$; moreover, it is $S'$-smooth, of finite presentation, and separated, since $\mathcal H$ is closed in the separated scheme $\mathcal A_{S'}\times_{S'}\Psi_{S'}$.

\textbf{Step 5: showing that $\mathcal N^{\alpha}$ is a test-N\'eron model.} To ease notation, let us write $S$ in place of $S'$, $D=\bigcup_{i\in\mathcal J} D_i$ for the strict normal crossing divisor $D\times_SS'$. Let $Z$ be a strictly henselian trait, with closed point $z$, and $g\colon Z\ra S$ a morphism transversal to $D$. Write $T$ for the strict henselization of $S$ at $z$ and $\mathcal E\subseteq \mathcal J$ for the subset of indices of components $D_i$ that contain $z$. Let also $\mathcal M/Z$ be the N\'eron model of $A\times_SZ$. The N\'eron mapping property gives a morphism $\mathcal N^{\alpha}_Z\ra \mathcal M$, which is an open immersion and induces an isomorphism between the fibrewise-connected components of identity, as they are both semi-abelian (\cref{A=N0}). Let $\Phi$ and $\Upsilon$ be the finite abelian groups of connected components of the fibres $\mathcal N^{\alpha}_z$ and $\mathcal M_z$ respectively. To show that $\mathcal N_Z\ra \mathcal M$ is an isomorphism, we only need to check that the induced morphism $\Phi \ra \Upsilon$ is an isomorphism. It is certainly injective; we claim that it is also surjective.

Let $Y, Y'$ be the character groups of the maximal tori of $\mathcal A_z, \mathcal A'_z$. By \cref{thm:grp_comps}, $\Upsilon$ is the kernel of the map $Y'\otimes \Q/\Z\to Y^{\vee}\otimes \Q/\Z$ induced by the monodromy pairing. Now for every $i\in \mathcal J$, let $X_i$ (resp. $X'_i$) be the character group of the maximal torus contained in $\mathcal A_{\zeta_i}$ (resp. $\mathcal A'_{\zeta_i})$, where $\zeta_i$ is the generic point of $D_i$. By the assumption of toric additivity, $Y$ is identified with $\bigoplus_{i\in\mathcal E}X_i$ and $Y'$ with $\bigoplus_{i\in\mathcal E}X'_i$.

The composition $\Psi=\underline\Psi(z)\twoheadrightarrow \underline\Psi(z)/\mathcal H(z)=\Phi \hookrightarrow \Upsilon$ fits into the commutative diagram with exact rows

\begin{center}
\begin{tikzcd}
0\ar[r] & \Psi\ar[r]\ar[d] & \bigoplus_{i\in\mathcal J}X'_i\otimes\Q/\Z \ar[r]\ar[d] & \bigoplus_{i\in\mathcal J}X_i^{\vee}\otimes\Q/\Z\ar[d]\\
0\ar[r] & \Upsilon\ar[r] & \bigoplus_{i\in\mathcal E}X'_i\otimes\Q/\Z \ar[r] & \bigoplus_{i\in\mathcal E}X_i^{\vee}\otimes\Q/\Z
\end{tikzcd}
\end{center}
where the rightmost horizontal maps are direct sums of the monodromy pairings. The fact that the bottom-right horizontal map is the sum of the monodromy pairings at the various $\zeta_i$, $i\in\mathcal E$ follows from \cref{prop:monodromy_transversal}.

Now, clearly the two rightmost vertical maps are surjective and admit compatible sections going upwards, obtained by the inclusion $\mathcal E\subset \mathcal J$; hence also the map $\Psi\ra \Upsilon$ admits a section $\Upsilon\to \Psi$, and is therefore surjective, which proves the claim.

We have shown that the object $\mathcal N^{\alpha}$ that we have constructed is a test-N\'eron model for $\mathcal A/S$; in particular, choosing $\beta\colon \Psi_{K_s}\to A_{K_s}[N]$ different from $\alpha$, induces a natural $S'$-isomorphism $\mathcal N^{\alpha}\to \mathcal N^{\beta}$, by \cref{test-NM_unique}. We may therefore drop the superscript $\alpha$.

\textbf{Step 7: descending $\mathcal N$ along $S'\ra S$.} For every geometric point $s$ of $S$, we have found an \'etale neighbourhood $S'\ra S$ and a test-N\'eron model $\mathcal N/S'$ over $S'$. Using uniqueness up to unique isomorphism of test-N\'eron models, their stability under \'etale base change, and effectiveness of \'etale descent for algebraic spaces, we obtain a smooth separated algebraic space of finite type $\til{\mathcal N}$ over $S$, and an isomorphism $\til{\mathcal N}\times_SU\ra A$. Because the property of being a test-N\'eron model is \'etale-local, $\til{\mathcal N}$ is itself a test-N\'eron model for $A$ over $S$. 
\end{proof}

\subsection{Test-N\'eron models and finite flat base change}\label{subsection:tNM_finite_flat}

In \citep{edix}, Edixhoven considers the case of an abelian variety $A_K$ over the generic point of a trait $S$, and a tamely ramified extension of traits $\pi\colon S'\ra S$ whose associated extension of fraction fields $K\ra K'$ is Galois. He considers the N\'eron model $\mathcal N/S$ of $A_K$ and the N\'eron model $\mathcal N'/S'$ of $A_{K'}$: after defining a certain equivariant action of $\Gal(K'|K)$ on the Weil restriction $\pi_*\mathcal N'$, he shows that $\mathcal N$ is naturally identified with the closed subgroup-scheme of $\Gal(K'/K)$-fixed points of $\pi_*\mathcal N'$.

In this subsection, we aim to show an analogous statement for test-N\'eron models of toric additive abelian schemes, over a base of higher dimension.

We let then $S$ be a noetherian, regular, strictly local scheme, of residue characteristic $p\geq 0$, $D=\cup_{i=1}^n\divv(t_i)$ a normal crossing divisor on $S$ (thus the $t_i$ are part of a system of regular parameters for $\O_S(S)$), $A$ an abelian scheme over $U=S\setminus D$, $\mathcal A/S$ a toric-additive semi-abelian scheme extending $A$. 

Consider now a finite flat cover $\pi\colon T\ra S$ of the form
$$T=\Spec \frac{\O_S(S)[X_1,\ldots,X_n]}{X_1^{m_1}-t_1,\ldots,X_n^{m_n}-t_n}$$

for some positive integers $m_1,\ldots,m_n$ all coprime to the residue characteristic $p$. Then $T$ is a regular strictly local scheme. We denote by $K'$ its field of fractions. The morphism $\pi$ is finite \'etale over $U$, tamely ramified over $S$, and the preimage via $\pi\colon T\ra S$ of $D$ is the normal crossing divisor $\pi^{-1}(D)=\cup_{i=1}^n\divv X_i$.

We have a commutative diagram
$$
\begin{tikzcd}
\Gal(K^t|K') \arrow[r, two heads]\arrow[hook]{d} & \pi_1^t(U_T) =\bigoplus_{i=1}^n\widehat\Z'(1) \arrow[d, xshift=-4.5ex, hook]\\
\Gal(K^t|K) \arrow[r, two heads] & \pi_1^t(U)=\bigoplus_{i=1}^n \widehat{\Z}'(1)
\end{tikzcd}
$$
where $K^t$ is the maximal tame extension of $K$, $\widehat \Z'(1)=\prod_{l\neq p}\Z_l(1)$, and the right vertical arrow is given by taking the $m_i$-th power on the $i$-th component. We will write $\pi_1(U)=\bigoplus_{i=1}^nI_i$ and identify $\pi_1^t(U_T)$ with its subgroup $\bigoplus_{i=1}^nm_iI_i$ (using additive notation).

The fraction field $K'$ of $T$ is an extension of $K$ of order $m_1\cdot m_2\cdot \ldots \cdot m_n$, and we write $G$ for the Galois group $\Gal(K'|K)$. As it has order $m_1\cdot m_2\cdot\ldots\cdot m_n$, we see that it is identified with the quotient $\pi_1^t(U_T)/\pi_1^t(U)=\bigoplus_{i=1}^n I_i/m_iI_i=\bigoplus_{i=1}^n \mu_{m_i}(U)$. 

By \cref{finiteflat_TA}, $\mathcal A\times_ST$ is still toric-additive. We follow the construction carried out in the proof of \cref{main_thm} to obtain a test-N\'eron model $\mathcal M/T$: to start with, we consider for each $i=1,\ldots,n$ the 
monodromy pairing for $\phi\colon X'_i\to X_i^{\vee}$ for $\mathcal A$ at the generic point of $\divv t_i$. The monodromy pairing $\phi_i'$ for $\mathcal A_T$ at the generic point of $\divv X_i$ is then given by composing $\phi_i$ with multiplication by $m_i$, 
$$\phi_i'=m_i\cdot \phi_i\colon X'_i\to X_i^{\vee}.$$
We construct the finite abelian group
$$\Psi'=\bigoplus_{i=1}^n\ker (\phi'_i\otimes\Q/\Z).$$

\begin{lemma}\label{claim:G-action_Psi}
There is a natural action of $G$ on $\Psi'$. 
\end{lemma}
\begin{proof}It is enough to define the action of $G$ on the $l^{\infty}$-torsion part $\Psi'\otimes\Q_l/\Z_l$, for each prime number $l$. 

We start by letting $G$ act trivially on the $p^{\infty}$-part. For $l\neq p$, we describe the natural action of $G$ on the $l^{\infty}$-torsion part $\Psi'[l^{\infty}]$. The $l$-monodromy pairing $\phi_i\otimes\Z_l$ is given by $$\sigma_i-1\colon \frac{T_lA(K^s)}{T_lA(K^s)^{I_i}}\to T_lA^{t_i}(K^{s})$$
where $\sigma_i$ is a generator of $I_i$. Hence $\phi_i'\otimes\Z_l=m_i\cdot\phi_i\otimes\Z_l=m_i(\sigma_i-1)=\sigma^{m_i}-1$. 

The $G$-action on $T_lA(K^s)$ is compatible with the Weil pairing; hence $I_i$ acts on $\ker(\phi'_i\otimes\Q_l/\Z_l)=\ker((\sigma_i^{m_i}-1)\otimes\Q_l/\Z_l)$. Since $\sigma_i^{m_i}$ is the generator of $m_iI_i$, the kernel is fixed by $m_iI_i$. It follows that there is a natural action of $I_i/m_iI_i$ on $\ker(\phi'_i\otimes\Q_l/\Z_l)$, hence a natural action of $G=\bigoplus I_i/m_iI_i$ on $\Psi'[l^{\infty}]$, hence a natural action of $G$ on $\Psi'=\bigoplus_l\Psi'[l^{\infty}]$.
\end{proof}

We write
$$\Psi=\bigoplus_{i=1}^n\ker(\phi_i\otimes\Q/\Z).$$
Clearly, $\Psi$ is a subgroup of $\Psi'$.
\begin{lemma}
$\Psi$ is equal to the subgroup $\Psi'^G$ of invariants with respect to the $G$-action which we defined in \cref{claim:G-action_Psi}.
\end{lemma}
\begin{proof}

Again, we only need check this on the $l^{\infty}$-parts.

When $l=p$, the claim says that $\Psi[p^{\infty}]=\Psi'[p^{\infty}]$, as $G$ acts trivially on the $p$-part of $\Psi'$. This is indeed true: as every $m_i$ is coprime to $p$, we have  for all $i=1,\ldots,n$, 
$$\ker(\phi_i\otimes\Q_p/\Z_p)=\ker(m_i\phi_i\otimes\Q_p/\Z_p).$$

If  $l\neq p$, $\ker(\phi_i\otimes\Q_l/\Z_l)$ consists of those elements of $\ker(\phi'_i\otimes\Q_l/\Z_l)$ that are in the kernel of $\sigma_i-1$, hence the $I_i/m_iI_i$-invariants, as we wanted to show.
\end{proof}

Next, we let $N=\ord(\Psi')$. Thank to the hypothesis of toric additivity, we can go carry out the construction of a test-N\'eron model as in the proof of \cref{main_thm}. We start by choosing a section $\alpha\colon\Psi'\ra A[N](K')$ of the natural surjection $A[N](K')\to \Psi'$. Then we write $H'$ for the image of $\Psi'\xrightarrow{(\alpha,\id)} A[N](K')\times\Psi'$ and $\mathcal H'$ for its schematic closure inside $\mathcal A_T\times_T\Psi'_T$. The fppf-quotient 
$$\mathcal M=\frac{\mathcal A_T\times_T \Psi'_T}{\mathcal H'}$$
is represented by a test-N\'eron model for $A_{U'}$ over $T$.

In order to compare $\mathcal M$ and $\mathcal N$, we will consider the Weil restriction of $\mathcal M$ via $\pi\colon T\ra S$, that is, the functor $\pi_*\mathcal M\colon (\Sch/S)\ra \Sets$ given by $(Y\ra S)\mapsto \mathcal M(Y\times_ST)$. Recall that we have an exact sequence of fppf-sheaves of abelian groups 
$$0\ra \mathcal H' \ra \mathcal A_T\times_T\Psi'_T\ra \mathcal M\ra 0.$$
As $\pi$ is a finite morphism, the higher direct images of $\pi$ for the fppf-topology vanish, and we have an exact sequence of fppf-sheaves
$$0\ra \pi_*\mathcal H'\ra \pi_*\mathcal A_T\times_S\pi_*\Psi'_T\ra \mathcal \pi_*\mathcal M \ra 0.$$

We claim that $\pi_*\mathcal M$ is representable by a scheme. By \citep[XI, 1.16]{raynaud}, semi-abelian schemes are quasi-projective, hence so is $\mathcal A_T\times_T\Psi'_T$. Clearly $\mathcal H'/T$ is quasi-projective as well. As $\pi\colon T\ra S$ is finite and flat, $\pi_*\mathcal H'$ and $\pi_*\mathcal A_T\times_S\pi_*\Psi'_T$ are schemes (see for example \citep[2.2]{edix}). Now, $\pi_*\mathcal H'/S$ is \'etale (\citep[4.9]{scheiderer}), and its intersection with the identity component of $\pi_*\mathcal A_T\times_S\pi_*\Psi'_T$ is trivial. Reasoning as in the proof of \cref{main_thm}, we conclude that $\pi_*\mathcal M$ has an open cover by schemes, hence it is a scheme.

We want to define an equivariant action of $G$ on $\pi_*\mathcal M\ra S$, where $G$ acts trivially on $S$. To do this, we let first $G$ act on $A_{K'}$ via the action of $G$ on $K'$. By \citep[1.3 pag.132]{deligne} the action of $G$ extends uniquely to an equivariant action on $\mathcal A_T\ra T$. We have also described an action of $G$ on $\Psi'$, which induces an equivariant action on $\Psi'_T\ra T$. We put together these actions to find an equivariant action of $G$ on $\mathcal A_T\times_T\Psi'_T\ra T$: clearly $H'$ is $G$-invariant, thus the same is true for its schematic closure $\mathcal H'$. Therefore the action of $G$ descends to an equivariant action of $G$ on $\mathcal M\ra T$.

To define the action of $G$ on $\pi_*\mathcal M$, we let $g\in G$ act on $\pi_*\mathcal M$ via the composition

$$\pi_*\mathcal M\times_ST\xrightarrow{(\id,g)}\pi_*\mathcal M\times_ST\ra \mathcal M\xrightarrow{g^{-1}}\mathcal M. $$

where the second arrow is given by the identity morphism $\pi_*\mathcal M\ra \pi_*\mathcal M$. This defines the desired equivariant action of $G$ on $\pi_*\mathcal M\ra S$. 

Consider the functor of fixed points $(\pi_*\mathcal M)^G\colon \Sch/S\ra \Sets$, $(Y\ra S)\mapsto \pi_*\mathcal M(Y)^G.$ Then $(\pi_*\mathcal M)^G$ is represented by a closed subgroup-scheme of $\pi_*\mathcal M$ \citep[3.1]{edix}; as the order of $G$ is invertible on $S$, $(\pi_*\mathcal M)^G$ is even smooth over $S$  (\citep[3.4]{edix}).

\begin{proposition} \label{weil}
Let $\mathcal N/S$ be a test-N\'eron model over $S$. There is a canonical closed immersion $\iota\colon \mathcal N\ra \pi_*\mathcal M$, which identifies $\mathcal N$ with the subgroup-scheme of fixed points $(\pi_*\mathcal M)^G$.

\end{proposition}
\begin{proof}
As usual, $\mathcal N$ can be constructed by choosing a section $\alpha\colon \Psi\ra A[N](K)$, constructing $H\subset A[N](K)\times \Psi$ and taking its closure $\mathcal H\subset \mathcal A\times_S\underline \Psi_S$, and finally the quotient. As the choice of $\alpha$ does not matter, we may assume it is the one obtained by restriction of the section $\Psi'\to A[N](K')$ used to construct $H'$. It follows that \begin{equation}\label{eq:H}
H=H'\cap (A(K)\times_K\Psi)
\end{equation}

By generalities on the Weil restriction \citep[pag. 198]{BLR}, the canonical morphism $\mathcal A\ra \pi_*\mathcal A_T$ is a closed immersion. The natural injection $\Psi\ra \Psi'$ gives a closed immersion $\mathcal A\times_S \Psi_S\ra \pi_*\mathcal A_T\times_S\Psi'_S=\pi_*(\mathcal A_T\times_T\Psi'_T)$. To show that it descends to a closed immersion $\mathcal N\ra \pi_*\mathcal M$, it is enough to show that 
\begin{equation}\label{Hrelation}
\pi_*\mathcal H'\cap (\mathcal A\times_S\Psi_S)=\mathcal H.
\end{equation} 

Notice that \cref{eq:H} realizes \cref{Hrelation} on the level of generic fibres. Now, $\pi_*\mathcal H'$ is \'etale over $S$, and it is a closed subscheme of $\pi_*\mathcal A_T\times_S\Psi'_S$. Hence, it is the schematic closure of its generic fibre, which is $H'$. Then, the intersection $\mathcal H^*:=\pi_*\mathcal H'\cap (\pi_*\mathcal A_T\times_S\Psi_S)$ is clearly still \'etale over $S$, and has generic fibre $H$. Thus $\mathcal H^*$ is the schematic closure of $H$ in $\pi_*\mathcal A_T\times_S\Psi_S$. On the other hand, $\mathcal H\ra \mathcal A\times_S\Psi_S\ra \pi_*\mathcal A_T\times_S\Psi_S$ is a closed immersion, and $\mathcal H$ is \'etale over $S$ and has generic fibre $H$. As $\mathcal H$ and $\mathcal H^*$ are both \'etale over $S$, have same generic fibre and are both closed subschemes of $\pi_*\mathcal A_T\times_S\Psi_S$, they are equal. Since $\mathcal H$ is contained in $\mathcal A\times_S\Psi_S$, so is $\mathcal H^*$ and we obtain \cref{Hrelation}. This proves that we have a closed immersion $\iota \colon\mathcal N\ra \pi_*\mathcal M$.

Now, the restriction of $\iota$ to the generic fibre is the closed immersion $A\ra \pi_*A_{K'}$, which identifies $A$ with $(\pi_*A_{K'})^G$. Both $(\pi_*\mathcal M)^G$ and $\mathcal N$ are both $S$-smooth closed subschemes of $\pi_*\mathcal M$, so both are equal to the schematic closure in $\pi_*\mathcal M$ of their own generic fibre. As they share the same generic fibre, they are equal. 
\end{proof}

\subsection{Test-N\'eron models are N\'eron models}
The objective of this subsection is to prove the following:
\begin{theorem} \label{TA->Neron}
Let $S$ be a locally noetherian, regular $\mathbb Q$-scheme, $D$ a normal crossing divisor on $S$, $\mathcal A/S$ a semiabelian scheme such that its restriction $\mathcal A_U$ to the open $U=S\setminus D$ is abelian and toric additive. Then $\mathcal A_U$ admits a N\'eron model over $S$.
\end{theorem}

In view of \cref{main_thm}, \cref{TA->Neron} is an immediate corollary of the following proposition:
\begin{proposition}\label{test=Neron}
Hypotheses as in \cref{TA->Neron}. Let $\mathcal N/S$ be a test-N\'eron model for $\mathcal A_U$ over $S$. Then $\mathcal N/S$ is a N\'eron model.
\end{proposition}

We will subdivide the proof of \cref{test=Neron} in two main steps (\cref{dim2,weakNMdimn}).

\begin{proposition} \label{dim2}
In the hypotheses of \cref{test=Neron}, assume $S$ has dimension $2$. Then the test-N\'eron model $\mathcal N/S$ is a weak N\'eron model for $\mathcal A_U$.
\end{proposition}
\begin{proof} 
Let $\sigma:U\ra A$ be a section; we want to show that it extends to a section $S\ra \mathcal N$, or equivalently, that the schematic closure $\overline{\sigma(U)}\subset \mathcal N$ is faithfully flat over $S$. The latter may be checked locally for the fpqc topology; hence, we may reduce to the case where $S$ is the spectrum of a complete, strictly henselian local ring. The normal crossing divisor $D$ has at most $2$ components; up to restricting $U$ we may assume they are exactly two, and that $D$ is the zero locus of $uv$, with $u,v$ regular parameters for $\Gamma(S,\mathcal O_S)$.

Notice that the schematic closure $\overline{\sigma(U)}$ may a priori fail to be flat only over the closed points of $S$, as $S\setminus\{s\}$ is of dimension $1$. By the flattening technique of Raynaud-Gruson (\citep[5.2.2]{raynaud-gruson}), there exists a blowing-up $\tilde S\ra S$, centered at $s$, such that the schematic closure of $\sigma(U)$ inside $\mathcal N_{\tilde S}$ is flat over $\tilde S$. Because $S$ has dimension 2, we can find a further blow-up $S'\ra \tilde S$ such that the composition $S'\ra S$ is a composition of finitely many blowing-ups, each given by blowing-up the ideal of a closed point with its reduced structure. It follows that the exceptional fibre $E\subset S'$ of $S'\ra S$ is a chain of projective lines meeting transversally. Let $\Sigma\subset \mathcal N_{S'}$ be the schematic closure of $\sigma(U)$. The morphism $\Sigma\ra S'$ is flat, but may a priori not be surjective. At this point we only know that the image of $\Sigma$ contains $S'\setminus E$. 

We claim that $\Sigma\ra S'$ is surjective. 
Let $p\in E$. It's easy to show that there exists some strictly henselian trait $Z$ with closed point $z$ and a closed immersion $Z\ra S'$ mapping $z$ to $p$ and such that $Z$ meets $E$ transversally. We call $L$ the field of fractions of $\O_Z(Z)$. The section $\sigma\colon U\ra A$ restricts to a section $\sigma_L\colon \Spec L\ra A_L$; to establish the claim, it suffices to show that $\sigma_L$ extends to a section $Z\ra \mathcal N_Z$.  We consider the composition $\phi:Z\ra S'\ra S$ and the pullbacks $\phi^*(u),\phi^*(v)\in \O_Z(Z)$. Let $m,n\in \mathbb Z_{\geq 1}$ be their respective valuations. Now let $\pi\colon T\ra S$ be the finite flat morphism given by extracting an $m$-root of $u$ and an $n$-root of $v$, that is, 
$$T=\Spec \frac{\O_S(S)[x,y]}{x^m-u,y^n-v}.$$
Then $T$ is itself the spectrum of a regular, strictly henselian local ring and the preimage $\pi^{-1}(D)$ is the zero locus of $xy$ and hence a normal crossing divisor. The pullback of $\mathcal A$ via $T\ra S$ is still toric additive (\cref{finiteflat_TA}) and therefore we can construct a test-N\'eron model $\mathcal M/T$. Using the hypothesis that the base $S$ has equicharacteristic zero, we may apply the results of \cref{subsection:tNM_finite_flat}: writing $X=\pi_*\mathcal M$ for the Weil restriction along $\pi$ and $G:=\Aut_S(T)=\mu_m\oplus\mu_n$, we have by \cref{weil} that $X^G=\mathcal N$.

Now, as $Z$ is a strictly henselian and of residue characteristic zero, $\O(Z)$ contains all roots of elements of $\O(Z)^{\times}$, and we can find uniformizers $t_u,t_v\in \O_Z(Z)$ such that $t_u^m=\phi^*(u)$ and $t_v^n=\phi^*(v)$. These elements give us a lift of $\phi \colon Z\ra S$ to $\psi\colon Z\ra T$. Then $\psi$ is a closed immersion meeting $f^{-1}(D)$ transversally. This means that the base change $\mathcal M_{Z}/Z$ is a N\'eron model of its generic fibre. Consider the section $\sigma_L\colon \Spec L\ra A_L$. Composing it with the closed immersion $A_L=(\pi_*\mathcal M)_L^G\hookrightarrow (\pi_*\mathcal M)_L$ gives, by definition of Weil restriction, a morphism $\Spec L\times_ST\ra \mathcal M_L$.  Precomposing with $(\id,\psi)\colon\Spec L\ra \Spec L\times_ST$, we obtain a section $\til{\sigma}_L\colon \Spec L\ra \mathcal M_L$.  As $\mathcal M_Z/Z$ is a N\'eron model of its generic fibre, $\til{\sigma}_L$ extends uniquely to a section $Z\ra \mathcal M_Z$. This gives us a morphism of $T$-schemes $Z\ra \mathcal M$ and by composition a $T$-morphism $Z\times_ST\ra Z\ra \mathcal M$, that is, a section $m\in X(Z)$ of the Weil restriction. Notice that the generic fibre of $m$ is $\sigma_L$, which lands in the part of $X$ fixed by $G$; as $X^G=\mathcal N$ is a closed subscheme of $X$ we deduce that $m$ lands inside $\mathcal N$. So $m\in \mathcal N(Z)$ is the required extension of $\sigma_L$ and we win.

As $\Sigma\ra S'$ is faithfully flat, separated and birational, it is an isomorphism. Hence $\sigma\colon U\ra A$ extends to a section $\sigma'\colon S'\ra \mathcal N_{S'}$. We are going to show that $\sigma'$ descends to a section $\theta:S\ra \mathcal N$. The restriction of $\sigma'$ to $E$ maps a connected chain of projective lines to a connected component of $\mathcal N_s$ (where $s$ is the closed point of $S$). Every connected component of $\mathcal N_s$ is isomorphic to the semi-abelian variety $\mathcal N^0_s$, hence does not contain projective lines. It follows that $\sigma'_{|E}$ is constant and that it descends to a morphism $\Spec k(s)\ra \mathcal N_s$. Let $\mathcal J$ be the ideal sheaf of the exceptional fibre $E\subset S'$ and define $S'_n\subset S'$ to be the closed subscheme defined by $\mathcal J^{n+1}$ for every $n\geq 0$. Similarly let $S_n:=\Spec \O_S(S)/\m^{n+1}$, where $\m$ is the maximal ideal of $\O_S(S)$. We have shown that $\sigma'_{|E}:S'_0\ra \mathcal N$ descends to a morphism $\theta_0\colon S_0\ra \mathcal N$. Now, by smoothness of $\mathcal N$, every morphism $S_{j-1}\ra \mathcal N$ admits a lift $S_j\ra \mathcal N$; the set of such lifts is given by $H^0(S_0,\Omega^1_{\mathcal N_{S_0}/S_0}\otimes_{\O_{S_0}}\m^j/\m^{j+1}).$ The canonical morphism 
$$H^0(S_0,\Omega^1_{\mathcal N_{S_0}/S_0}\otimes_{\O_{S_0}}\m^j/\m^{j+1})\ra H^0(S'_0,\Omega^1_{\mathcal N_{S'_0}/S'_0}\otimes_{\O_{S'_0}}\mathcal J^j/\mathcal J^{j+1})$$
is an isomorphism, due to the fact that the space of global sections of $\mathcal J^j/\mathcal J^{j+1}=\O_{S'_0}(j)$ is equal to $\m^j/\m^{j+1}$. Thus the set of liftings of $\alpha\in\Hom_S(S_j,\mathcal N)$ to $\Hom_S(S_{j+1},\mathcal N)$ is naturally in bijection with the set of liftings of $\alpha_{|S'_j}\in\Hom_S(S'_j,\mathcal N)$ to $\Hom_S(S'_{j+1},\mathcal N)$. The reductions modulo $\mathcal J^j$ of $\sigma'\colon S'\ra \mathcal N_{S'}$ provides a compatible set of liftings of $\sigma'_{|S'_0}$, and therefore a compatible set of liftings of $\theta_0$; which in turn by completeness of $S$ yield the desired morphism $S\ra \mathcal N$.
\end{proof}

The next step is extending the result to the case of $\dim S>2$. 

\begin{proposition} \label{weakNMdimn}
In the hypotheses of \cref{test=Neron}, $\mathcal N/S$ is a weak N\'eron model. 
\end{proposition}
\begin{proof}
As in the proof of \cref{dim2}, we may assume that $S$ is the spectrum of a complete strictly henselian local ring. We proceed by induction on the dimension of $S$. If the dimension is $1$, the statement is clearly true, and the case of dimension $2$ is the statement of \cref{dim2}. So we let $n\geq 3$ be the dimension of $S$ and we suppose that the statement is true when $S$ has dimension $n-1$. Let $\sigma\colon U\ra A$ be a section. Because $V=S\setminus\{s\}$ has dimension $n-1$, and because $\mathcal A_V$ is still toric additive (\cref{TAsmooth}), $\sigma$ extends to $\sigma\colon V\ra \mathcal N_V$.

Next, we cut $S$ with a hyperplane $H$ transversal to all the components of the normal crossing divisor $D$, but paying attention to choosing $H$ so that $D\cap H$ (with its reduced structure) is still a normal crossing divisor on $H$. This is always possible: consider a system of regular parameters $u_1,u_2,\ldots,u_n$ for $S$ such that $D$ is the zero locus of $u_1u_2\cdots u_r$ for some $r\leq n$; then $H$ can be chosen to be, for example, the hypersurface cut by $u_1-u_n$. Because $H$ is transversal to $D$, it is clear that the base change $\mathcal N_H/H$ is still a test-N\'eron model. By our inductive assumption on the dimension of the base, $\sigma_{|H}\colon H\cap U\ra A$ extends to $\theta_0\colon H\ra \mathcal N$. Now we would like to put together the data of $\sigma$ and $\theta_0$ to extend $\sigma\colon V\ra \mathcal N_V$ to a section $\theta\colon S\ra \mathcal N$. Let $\mathcal J\subset \O_S$ be the ideal sheaf of $H$ and for every $j\geq 1$ define $S_j$ to be the closed subscheme cut by $\mathcal J^{j+1}$. We have a morphism $\theta_0\colon H=S_0\ra \mathcal N$. By smoothness of $\mathcal N$, there exists for every $j\geq 0$ a lifting of $\theta_0$ to an $S$-morphism $S_j\ra \mathcal N$. The set of liftings of an $S$-morphism $S_{j-1}\ra \mathcal N$ to an $S$-morphism $S_{j}\ra \mathcal N$ is given by the global sections of the locally-free sheaf $\mathcal F:=\Omega^1_{\mathcal N/S}\otimes \mathcal J^j/\mathcal J^{j+1}$ on $S_0$. Because $\dim S_0\geq 2$ and $V=S\setminus\{s\}$, we have $H^0(S_0,\mathcal F)=H^0(V\cap S_0,\mathcal F_V)$, and the latter parametrizes liftings of morphisms $S_{j-1}\cap V\ra \mathcal N$ to $S_{j}\cap V\ra \mathcal N$. The section $\sigma\colon V\ra \mathcal N_V$ gives a compatible choice of lifting for every $j\geq 0$, and we get by completeness of $S$ a morphism $S\ra \mathcal N$ agreeing with $\sigma$ on $V$, as we wished. 
\end{proof}

We can now conclude the proof of \cref{test=Neron}.
\begin{proof}[Proof of \cref{test=Neron}]
Let $T\ra S$ be a smooth morphism; then $\mathcal A_T/T$ is toric-additive by \cref{TAsmooth} and the base change $\mathcal N_T/T$ is a test-N\'eron model. Now, given $\sigma_U\colon T_U\ra \mathcal A$, we obtain a section $T_U\ra \mathcal A\times_UT_U$, which by \cref{weakNMdimn} extends to a section $T\ra \mathcal N_T$. The latter is the datum of an $S$-morphism $\sigma\colon T\ra \mathcal N$ extending $\sigma_U$.
\end{proof}

We give a corollary of \cref{TA->Neron}.
\begin{corollary}
Let $S$ be a connected, locally noetherian, regular $\mathbb Q$-scheme, $D$ a regular divisor on $S$, $A$ an abelian scheme over $U=S\setminus D$ extending to a semi-abelian scheme $\mathcal A/S$. Then $A$ admits a N\'eron model over $S$.
\end{corollary}
\begin{proof}
At every geometric point $s$ of $S$, $D$ has only one irreducible component. It follows that $\mathcal A/S$ is toric-additive and we conclude by \cref{TA->Neron}.
\end{proof}

\begin{remark}
It is not clear whether the hypothesis that $S$ has equicharacteristic zero can be removed, at least partially, from \cref{TA->Neron}. In the course of the proof, the characteristic assumption is used only in the proof of \cref{dim2}, where we apply \cref{weil} on descent of N\'eron models along tamely ramified covers of discrete valuation rings. Removing the characteristic assumption would involve generalizing \cref{weil} to the case of wildly ramified covers.
\end{remark}

\section{A converse statement}\label{section5}
We conclude with a partial converse statement to \cref{TA->Neron}. 
\begin{theorem}\label{thm:converse}
Let $S$ be a locally noetherian, regular $\mathbb Q$-scheme, $D$ a normal crossing divisor on $S$, $\mathcal A/S$ a semiabelian scheme such that its restriction $A=\mathcal A_U$ to $U=S\setminus D$ is abelian. Assume that $A/U$ and its dual $A'/U$ admit N\'eron models $\mathcal N/S$ and $\mathcal N'/S$ which are at the same time also test-N\'eron models.

Then $\mathcal A/S$ is toric additive.
\end{theorem}
\begin{proof}
We may reduce to the case where $A$ is principally polarized. Indeed, $A^4\times A'^4$ is principally polarized (by Zarhin's trick) and still admits a N\'eron model $\mathcal N^4\times \mathcal N'^4$ which is a test-N\'eron model. If the unique semiabelian extension $\mathcal A^4\times \mathcal A'^4$ is toric additive, then so is $\mathcal A$.

Next, we may restrict to the case where $S$ is strictly henselian, by \cref{NMlocaliz} and the fact that toric additivity is defined at strict henselizations. Now we argue by induction on the number of irreducible components of $D$. If $D$ is empty, $\mathcal A/S$ is automatically toric additive. We let then $n>0$ be the number of components of $D$ and assume that the result is true for divisors with less than $n$ irreducible components. 

We write $\zeta_1,\ldots,\zeta_n$ for the generic points of $D$, and $X_1,\ldots,X_n$ for the groups of characters of the maximal tori of the fibres of $\mathcal A$ over the $\zeta_i$'s. Let also $\zeta$ be the generic point of $D_2\cap D_3\cap\ldots\cap D_n$, and $Y$ the corresponding group of characters. Let $S'\to S$ be a strict henselization at a geometric point lying over $\zeta$. By inductive hypothesis, $\mathcal A_{S'}/S'$ is toric additive. That is, the purity map at $\zeta$, $p_{\zeta}\colon Y\to X_2\oplus\ldots\oplus X_n$, is an isomorphism. To prove toric additivity of $\mathcal A/S$, that is, that the purity map $X\to X_1\oplus \ldots X_n$ is an isomorphism, it suffices therefore to show that the natural morphism
$$X\to X_1\oplus Y$$
is an isomorphism. Let us write $Y_1:=X_1$, $Y_2:=Y$.
We make a choice of bases for $X, Y_1, Y_2$, so that these groups are identified with their duals with values in $\Z$. We consider the diagram of free finitely generated abelian groups
\begin{equation}\label{eq:monodromy_matrices}
\begin{tikzcd}X \ar[r, "A"] & Y_1 \oplus Y_2 \ar[r, "\Psi"] & Y_1\oplus Y_2 \ar[r, "A^t"] & X
\end{tikzcd}
\end{equation}
where:
\begin{itemize}
\item the injective matrix $A$ is equal to $\begin{pmatrix} P \\ Q \end{pmatrix}$, with $P$ and $Q$ the surjective matrices corresponding to the two specialization maps.
\item $\Psi=\begin{pmatrix}
\Psi_1 & 0\\
0 & \Psi_2
\end{pmatrix}$, where $\Psi_1\colon Y_1\to Y_1'^{\vee}=Y_1$ is the monodromy pairing at $\zeta_1$, and $\Psi_2\colon Y_2\to Y_2$ is the direct sum of the monodromy pairings at $\zeta_2,\ldots,\zeta_n$. By \cref{thm:monodromy_pairing}, $\Psi_1$ and $\Psi_2$ are symmetric, positive definite.
\item $A^t$ is the transpose of $A$.
\end{itemize}

Tensoring with $\Q/\Z$ we obtain a new diagram
\begin{equation}\label{eq:monodromy_matrices_res}
\begin{tikzcd}\bar X \ar[r, "\bar A"] & \bar Y_1 \oplus \bar Y_2 \ar[r, "\bar \Psi"] & \bar Y_1\oplus \bar Y_2 \ar[r, "\bar A^t"] & \bar X
\end{tikzcd}
\end{equation}

Let $Z\to S$ be a transversal trait mapping the closed point to the closed point. The composition \eqref{eq:monodromy_matrices} is the monodromy pairing for $\mathcal A_Z$, by \cref{prop:monodromy_transversal}. By \cref{thm:grp_comps} the component group $\Upsilon$ of the N\'eron model of $\mathcal A_Z$ over $Z$ is $\ker(\bar A^t\bar\Psi\bar A)=\coker(A^t\Psi A)$. On the other hand, the component group $\Phi$ of $\mathcal N/S$ at the closed point is contained in $\ker(\bar P^t\bar\Psi_1\bar P)\cap \ker(\bar Q^t\bar \Psi_2\bar Q)$, by \cref{prop:grp_comps}. Our hypotheses tell us that $\Psi=\Phi$. 

\begin{claim}
$\ker(\bar P^t\bar\Psi_1\bar P)\cap \ker(\bar Q^t\bar \Psi_2\bar Q)=\coker(A^t\Psi)$
\end{claim}
\begin{proof}
As $P$ and $Q$ are surjective homomorphisms of free abelian groups, $P^t$ and $Q^t$ are injective with free cokernel. If tollows that $\bar P^t$ and $\bar Q^t$ are injective. The statement becomes $\coker(A^t\Psi)=\ker(\bar\Psi_1\bar P)\cap \ker(\bar \Psi_2\bar Q)=\ker(\bar\Psi\bar A)$. Now, $\Psi A$ is injective; it's a simple check that an injective morphism $f$ of free finitely generated $\Z$-modules satisfies $\coker(f^t)=\ker(f\otimes\Q/\Z)$. This proves the claim.
\end{proof}

The canonical isomorphism $\Psi=\Phi$ gives us a natural injective morphism of finite abelian groups $\coker(A^t\Psi)\into \coker(A^t\Psi A)$. Clearly, there is also a natural quotient map between the two, and it follows that $$\coker(A^t\Psi)=\coker(A^t\Psi A).$$ 

This can be restated by saying $\im(A^t\Psi)=\im(A^t\Psi A)$. As $A^t\Psi A$ is injective, we obtain a unique group homomorphism $\theta\colon Y_1\oplus Y_2\to X$, such that 
\begin{equation}\label{equation_boh}
A^t\Psi A\theta=A^t\Psi.
\end{equation} We write $\theta=(\theta_1,\theta_2)$ with $\theta_1\colon Y_1\to X$ and $\theta_2\colon Y_2\to X$. Because $A^t\Psi A$ is injective, we immediately get 
\begin{equation}\label{eq:theta}\theta A=\id_X.
\end{equation}

We consider the two endormorphisms $\chi_1=\theta_1P$ and $\chi_2\colon\theta_2Q$ of $X$. Equation \eqref{eq:theta} gives us 
\begin{equation}\label{chiechi}
\chi_1+\chi_2=\id_X.\end{equation}

Restricting \eqref{equation_boh} to $Y_1$, we get $P^t\Psi_1=A^t\Psi A\theta_1$. Precomposing with $P$, we find 
\begin{equation}\label{eq:bohboh}
P^t\Psi_1P=A^t\Psi A\chi_1=P^t\Psi_1 P\theta_1+Q^t\Psi_2 Q\theta_1.\end{equation}

Together with \eqref{chiechi}, this implies that 
\begin{equation}\label{MchiNchi}
P^t\Psi_1P\chi_2=Q^t\Psi_2Q\chi_1.
\end{equation}

Let $M=P^t\Psi_1P$ and $N=Q^t\Psi_2Q$. They are both positive semidefinite matrices, and we have the system given by \eqref{chiechi} and \eqref{MchiNchi}
$$
\begin{cases}
M\chi_2=N\chi_1\\
\chi_1+\chi_2=1
\end{cases}
$$
It follows that $M=(M+N)\chi_2$. 

We will now look at our maps over the real numbers $\R$. For a homomorphism $f$ between free, finitely generated abelian groups, we denote by $\widetilde f$ the induced linear map of $\R$-vector spaces.

As $A^t\Psi A$ is symmetric positive definite, $\widetilde{A^t}\widetilde{\Psi}\widetilde A$ is invertible in $\End(X\otimes \R)$. It follows from \eqref{eq:bohboh} that $\wt{\chi_1}$ is symmetric, hence, by the spectram theorem, diagonalizable with real eigenvalues. The same holds for $\wt \chi_2$.

\begin{claim}
Both $\chi_1$ and $\chi_2$ are idempotent.
\end{claim}
\begin{proof}
It suffices to check that $\wt{\chi}_1$ and $\wt{\chi}_2$ are idempotent. As they are diagonalizable, this amounts to showing that each of their eigenvalues is either $0$ or $1$.
By contradiction, assume there exists $c\in \R\setminus\{0,1\}$  eigenvalue for $\chi_2$, and let $v\in X\otimes \R$ be a non-zero eigenvector for $c$. Then 
$$v^t\wt Mv=v^t(\wt M+\wt N)\chi_2v=v^t(\wt M+\wt N)cv.$$
Letting $\alpha=v^t\wt Mv\geq 0$ and $\beta=v^t\wt Nv\geq 0$, we find $\alpha=(\alpha+\beta)c$. As $c\neq 0,1$, we see that $\alpha$ and $\beta$ are both non-zero. It follows that $0<c<1$. In particular, all eigenvalues of $\chi_2$ are contained in the interval $[0,1]$. However, $\chi_2$ is an integer-valued matrix; it follows that its characteristic polynomial  $f_{\chi_2}(t)\in \Z[t]$ is equal to $t^mg(t)$ for some $g(t)\in \Z[t]$ with $g(0)\neq 0$. Moreover, $g(0)$ is the product of the non-zero eigenvalues of $\chi_2$. As one of them, $c$, is smaller than $1$, it follows that $0<g(0)<1$. This is absurd, since $g(0)\in \Z$.
\end{proof}

It follows from the claim that $\chi_1\chi_2=\chi_1(1-\chi_1)=\chi_1-\chi_1^2=0$. Similarly $\chi_2\chi_1=0$. This implies that $$X=\ker \chi_1\oplus \ker \chi_2.$$ Indeed, any $x\in X$ can be written as $x=\chi_1(x)+\chi_2(x)\in \ker \chi_1+\ker\chi_2$, and the decomposition is unique since $\ker\chi_1\cap \ker\chi_2=0$. 

Now, the equation $P^t\Psi_1=A^t\Psi A\theta_1$ tells us that $\theta_1$ is injective (and similarly $\theta_2$). Hence, $\ker\chi_1=\ker P$, and $\ker \chi_2=\ker Q$; in particular $X=\ker P\oplus \ker Q$. But now we are done: $P$ induces an isomorphism $X/\ker P\cong \ker Q\to Y_1$ and $Q$ induces an isomorphism $X/\ker Q\cong \ker P\to Y_2$. Hence $A$ is an isomorphism.
\end{proof}

\cleardoublepage
\addcontentsline{toc}{section}{Bibliography}

\bibliographystyle{alpha}

\begin{thebibliography}{{Ore}18}

\bibitem[And01]{Andreatta2001}
Fabrizio Andreatta.
\newblock {\em On Mumford's Uniformization and Neron Models of Jacobians of
  Semistable Curves over Complete Rings}, pages 11--126.
\newblock Birkh{\"a}user Basel, Basel, 2001.

\bibitem[BLR90]{BLR}
Siegfried Bosch, Werner L{\"u}tkebohmert, and Michel Raynaud.
\newblock {\em N\'eron Models}, volume~21 of {\em Ergebnisse der {M}athematik
  und ihrer {G}renzgebiete}.
\newblock Springer-Verlag, 1990.

\bibitem[Cap08]{caporaso}
Lucia Caporaso.
\newblock N\'eron models and compactified {P}icard schemes over the moduli
  stack of stable curves.
\newblock {\em Amer. J. Math.}, 130(1):1--47, 2008.

\bibitem[Del85]{deligne}
Pierre Deligne.
\newblock Le lemme de {G}abber.
\newblock {\em Ast\'erisque}, 127:131--150, 1985.
\newblock Seminar on arithmetic bundles: the Mordell conjecture (Paris,
  1983/84).

\bibitem[Edi92]{edix}
Bas Edixhoven.
\newblock N{\'e}ron models and tame ramification.
\newblock {\em Compositio Mathematica}, 81(3):291--306, 1992.

\bibitem[FC90]{faltings1990degeneration}
G.~Faltings and C.L. Chai.
\newblock {\em Degeneration of Abelian Varieties}.
\newblock A Series of modem surveys in mathematics. Springer-Verlag, 1990.

\bibitem[GR71]{raynaud-gruson}
L.~Gruson and M.~Raynaud.
\newblock Crit\`eres de platitude et de projectivit{\'e}. {T}echniques de
  ``platification'' d'un module.
\newblock {\em Inventiones mathematicae}, 13:1--89, 1971.

\bibitem[Gro71]{SGA1}
Alexander Grothendieck.
\newblock {\em Rev\^etements \'etales et groupe fondamental (SGA 1)}, volume
  224 of {\em Lecture notes in mathematics}.
\newblock Springer-Verlag, 1971.

\bibitem[GRR72]{SGA7}
Alexander Grothendieck, Michel Raynaud, and Dock~Sang Rim.
\newblock {\em Groupes de monodromie en g\'eom\'etrie alg\'ebrique. {I}}.
\newblock Lecture Notes in Mathematics, Vol. 288. Springer-Verlag, 1972.
\newblock S{\'e}minaire de G{\'e}om{\'e}trie Alg{\'e}brique du Bois-Marie
  1967--1969 (SGA 7 I).

\bibitem[Hol17]{holmes}
David Holmes.
\newblock N{\'e}ron models of jacobians over base schemes of dimension greater
  than 1.
\newblock {\em To appear in Journal f{\"u}r die reine und angewandte
  Mathematik}, 2017.

\bibitem[MB85]{pinceaux}
Laurent Moret-Bailly.
\newblock Pinceaux de vari\'et\'es ab\'eliennes.
\newblock {\em Ast\'erisque}, (129):266, 1985.

\bibitem[{Ore}18]{Orecchia}
Giulio {Orecchia}.
\newblock {A criterion for existence of N\textbackslash'eron models of
  jacobians}.
\newblock {\em arXiv e-prints}, page arXiv:1806.05552, Jun 2018.

\bibitem[Ray70]{raynaud}
Michel Raynaud.
\newblock {\em Faisceaux amples sur les sch\'emas en groupes et les espaces
  homog\`enes}.
\newblock Lecture Notes in Mathematics, Vol. 119. Springer-Verlag, Berlin-New
  York, 1970.

\bibitem[Sch94]{scheiderer}
Claus Scheiderer.
\newblock {\em Real and \'etale cohomology}, volume 1588 of {\em Lecture Notes
  in Mathematics}.
\newblock Springer-Verlag, Berlin, 1994.

\bibitem[{Sta}16]{stacks}
The {Stacks Project Authors}.
\newblock \itshape {S}tacks {P}roject.
\newblock \url{http://stacks.math.columbia.edu}, 2016.

\end{thebibliography}

\end{document}